\documentclass[journal,twoside,web]{ieeecolor}
\usepackage{generic}
\usepackage{cite}
\usepackage{amsmath,amssymb,amsfonts}
\usepackage{algorithmic}
\usepackage{graphicx}
\usepackage{algorithm,algorithmic}
\usepackage{hyperref}
\usepackage{textcomp}
\usepackage{latexsym,theorem}
\usepackage{subfig}
\usepackage{xcolor}
\usepackage{booktabs}

\newcommand{\tr}{\top}

\newcommand{\Rset}{\mathbb{R}}

\newcommand{\Cset}{\mathbb{C}}
\newcommand{\Nset}{\mathbb{N}}

\newcommand{\Kset}{\mathbb{K}}

\newcommand{\cA}{\mathcal{A}}

\newcommand{\cC}{\mathcal{C}}

\newcommand{\cN}{\mathcal{N}}
\newcommand{\cM}{\mathcal{M}}

\newcommand{\cV}{\mathcal{V}}

\newcommand{\cH}{\mathcal{H}}

\newcommand{\cK}{\mathcal{K}}

\newcommand{\bs}{\mathbf{s}}

\newcommand{\col}{\operatorname{col}}

\newcommand{\Hx}{\mathcal{H}_x}
\newcommand{\Hu}{\mathcal{H}_u}
\newcommand{\Hprod}{\mathcal{H}}
\newcommand{\mux}{\mu_x}
\newcommand{\muu}{\mu_u}
\newcommand{\Sigx}{\Sigma_x}
\newcommand{\Sigu}{\Sigma_u}

\newcommand{\Ltwo}{L^2}
\newcommand{\norm}[1]{\left\|#1\right\|}
\newcommand{\inn}[2]{\left\langle #1,\, #2 \right\rangle}
\newcommand{\pushfwd}{F_*}

{\theorembodyfont{\slshape}\newtheorem{theorem}{Theorem}[section]}
{\theorembodyfont{\slshape}\newtheorem{proposition}[theorem]{Proposition}}
{\theorembodyfont{\slshape}}
{\theorembodyfont{\slshape}}
{\theorembodyfont{\slshape}\newtheorem{corollary}[theorem]{Corollary}}
{\theorembodyfont{\upshape}\newtheorem{definition}[theorem]{Definition}}
{\theorembodyfont{\upshape}}
{\theorembodyfont{\upshape}\newtheorem{remark}[theorem]{Remark}}
{\theorembodyfont{\upshape}\newtheorem{assumption}[theorem]{Assumption}}

\def\BibTeX{{\rm B\kern-.05em{\sc i\kern-.025em b}\kern-.08em
    T\kern-.1667em\lower.7ex\hbox{E}\kern-.125emX}}
\begin{document}
\title{From Product Hilbert Spaces to the\\ Generalized Koopman Operator and the Nonlinear Fundamental Lemma}
\author{Mircea Lazar, \IEEEmembership{Senior Member, IEEE}
\thanks{M. Lazar is with the Control Systems Group, Electrical Engineering Department, Eindhoven University of Technology, De Groene Loper 19, 5612 AP, Eindhoven, The Netherlands (e-mail: m.lazar@tue.nl).}}

\maketitle

\begin{abstract}
The generalization of the Koopman operator to systems with control input and the derivation of a nonlinear fundamental lemma are two open problems that play a key role in the development of data-driven control methods for nonlinear systems. In this paper we derive a novel solution to these problems based on basis functions expansion in a product Hilbert space constructed as the tensor product between the Hilbert spaces of the state and input observable functions, respectively. We identify relaxed invariance conditions that guarantee existence of a bounded linear operator, i.e., the generalized Koopman operator, from the constructed product Hilbert space to the Hilbert space corresponding to the lifted state propagated forward in time. Compared to classical Koopman invariance conditions, measure preservation is not required. Moreover, we derive a nonlinear fundamental lemma by exploiting the constructed exact infinite-dimensional bilinear Koopman representation and Hankel operators. The effectiveness of the developed generalized Koopman embedding is illustrated on the Van der Pol oscillator and in predictive control of a soft-robotic manipulator model.
\end{abstract}


\section{Introduction}
\label{sec:1}
\IEEEPARstart{T}{he} development of data-driven analysis and control methods for dynamical systems is currently fueled by the merging of machine learning and control design methods, see, e.g., \cite{Koopman_book, Markovsky_survey}, which can potentially impact all control application domains in our society. The main challenge in developing data-driven analysis and control methods lies in dealing with the inherently complex and nonlinear dynamics of real-world systems.  Two of the most adopted approaches in data-driven analysis and control of nonlinear systems are the Koopman operator modeling approach \cite{Koopman1931, Koopman_book} and the behavioural modeling approach based on Willems' fundamental lemma \cite{WillemsRapisarda2005, Markovsky_survey}. In what follows, we provide a non-exhaustive overview of recent advances and open problems in these two areas and we show that they share a common challenge, which will be addressed in this paper.

\begin{figure}[!t]
\centerline{\includegraphics[width=0.7\columnwidth]{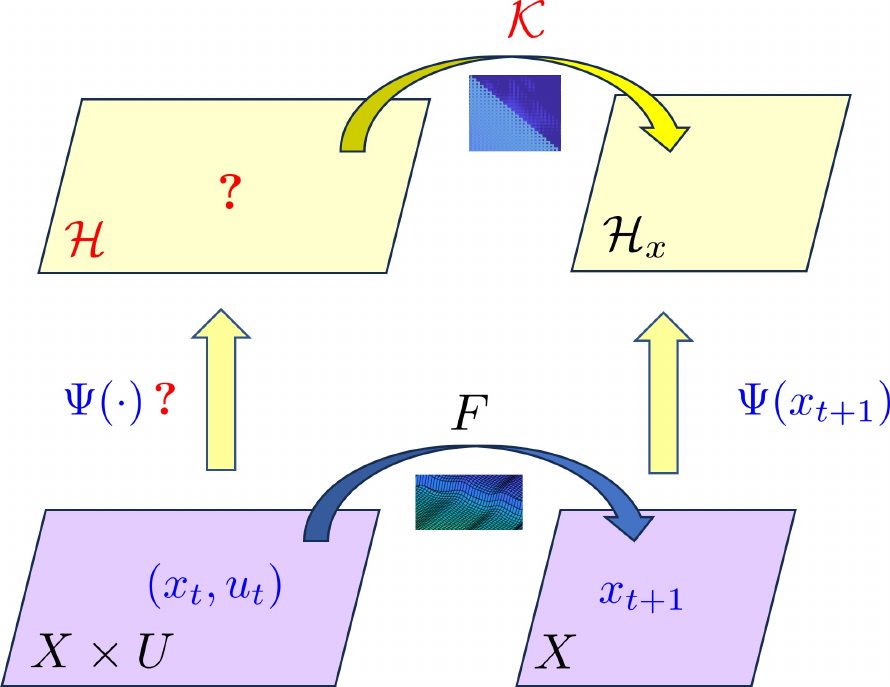}}
\caption{Illustration of the Koopman operator modeling approach for dynamical systems with control input.}
\label{fig1}
\end{figure}

The Koopman operator modeling approach \cite{Koopman1931, Koopman_book, Mezic_2021_review, Brunton_survey} constructs an infinite-dimensional linear representation of a nonlinear system that can be used for analysis or control. In the case of autonomous systems, that is, $x_{t+1}=F_0(x_t)$ with $F_0\,:\,X\rightarrow X$ a nonlinear map, this is typically done by means of a measure-preserving composition operator $(C_{F_0}\varphi)(x) = \varphi(F_0(x))$, i.e., the Koopman operator, which acts on observable functions $\varphi$ in a Hilbert space $\cH_x$. In practice, finite-dimensional approximations of the Koopman operator can be computed from data, see, e.g., the extended dynamic mode decomposition (EDMD) method \cite{Williams_EDMD_2015}. Recently, it was shown in \cite{Asada} that Koopman operators can be constructed using a direct encoding method based on orthonormal expansion in Hilbert spaces.
However, when the Koopman operator approach is applied to dynamical systems with control input (that is $x_{t+1}=F(x_t,u_t)$), as graphically illustrated in Figure~\ref{fig1}, a new problem emerges: how to construct the Hilbert space $\cH$ of observable functions on which the operator acts?

One of the first approaches to deriving a Koopman operator for systems with inputs originated in \cite{Korda_2018_Koop}, which augments the system state with an infinite-dimensional input sequence and constructs a measure-preserving composition operator for the augmented autonomous system. Therein it was shown how this approach can be used to derive approximate linear and bilinear Koopman representations. A similar approach was also employed in \cite{Brunto_Koopmanu}. Conditions for exact Koopman bilinearization (in the original control input) of input-affine nonlinear systems were derived in \cite{GoswamPaley2022, Bruder2020} and further researched, e.g., in \cite{Peitz_2024, strasser2026safedmd}. Other works derive a family of (autonomous) Koopman operators parameterized by the control input \cite{guo2025parametrickoopman, haseli2026modeling}, or derive linear-parameter-varying Koopman representations \cite{IACOB2024_Koopman}.

As an alternative to the Koopman modeling approach, Willems' fundamental lemma \cite{WillemsRapisarda2005, Markovsky_survey} aims at deriving a data-based system representation defined as a basis of recorded trajectories whose span coincides with all possible system trajectories under the assumption of persistency of excitation (PE). However, this result exploits the superposition principle, which is only valid for linear systems. Under a stronger PE condition, fundamental lemmas have been derived for state-space Hammerstein systems \cite{BerberichNL2020}, bilinear systems in \cite{Yuan_Cortes_Bilinear, Timm_bilinear} and input-output generalized bilinear systems in \cite{Markovsky_2023_Biliniear}. In \cite{Molodchyk_2024exploring, Lazar_SysDo_2024} it was shown that the fundamental lemma can be equivalently formulated using kernel functions and reproducing kernel Hilbert spaces (RKHS). Similarly as in the Koopman modeling approach, the key challenge lies in the construction of lifting basis functions such that in the lifted (Hilbert) space there exists a possibly infinite-dimensional, linear data-based system representation.

Motivated by the above state-of-the-art in Koopman-based and data-driven behavioral modeling of nonlinear systems, in this paper we develop a novel theoretical framework that uses \emph{product Hilbert spaces} for generalizing the Koopman operator and Willems' fundamental lemma to nonlinear systems with inputs. The \emph{main contributions} are as follows:
\begin{itemize}
\item A well-defined generalized Koopman composition operator for systems with inputs that acts on observable functions in a Hilbert space $\cH_x$ defined on the state-space $X$ and yields functions in the tensor product $\cH:=\cH_x\otimes\cH_u$ of Hilbert spaces defined on the state-space $X$ and input-space $U$, respectively;
 \item A coordinate representation of the generalized Koopman operator in Riesz bases as an operator $\cK : \ell^2\rightarrow \ell^2$ that yields an exact infinite-dimensional bilinear representation of the original nonlinear system;
\item A nonlinear fundamental lemma based on the infinite-dimensional bilinear Koopman representation (extended to systems with measured outputs) and Hankel operators.
\end{itemize}
The remainder of the paper is organized as follows. Preliminaries are introduced in  Section~\ref{sec:2}. The generalized Koopman operator based on product Hilbert spaces is presented in Section~\ref{sec:3}. The nonlinear fundamental lemma is derived in Section~\ref{sec:4}. Illustrative examples are shown in Section~\ref{sec:5} and conclusions are summarized in Section~\ref{sec:6}.

\section{Preliminaries}
\label{sec:2}
Let $\Cset$, $\Rset$ and $\Nset$ denote the sets of complex, real and natural numbers, respectively. Let $\Kset$ denote either $\Cset$ or $\Rset$. For  $q\in\Nset\cup\{\infty\}$ vectors (or vector-valued functions) $\{v_1,\ldots,v_q\}$, define $\col(v_1,\ldots,v_q):=[v_1^\tr,\ldots,v_q^\tr]^\tr$. For two matrices $A\in\Rset^{m\times n}, B\in\Rset^{p\times q}$, $A\otimes B\in\Rset^{mp\times nq}$ denotes their Kronecker product. For two vectors $a\in\Rset^n, b\in\Rset^m$, $a\otimes b\in\Rset^{nm}$ denotes their Kronecker vector product, i.e., for $a=[a_1\, a_2]^\top$, $b=[b_1\, b_2]^\top$, $a\otimes b=[a_1b_1\, a_1b_2 \, a_2b_1 \, a_2b_2]^\top$. For two matrices $A\in\Rset^{n\times m}$, $B\in\Rset^{p\times m}$, $A\odot B\in\Rset^{np\times m}$ denotes their Khatri-Rao product, which consists of the column-wise Kronecker product. For a full-row rank matrix $A\in\Rset^{m\times n}$, $A^\dagger=A^\top(AA^\top)^{-1}$ denotes its right Moore-Penrose inverse. For any sequence $\bs:=\{s_0,s_1,\ldots\}$ with $s_i\in\Rset^{n_s}$, define the truncated augmented vector $\bs_{[i,j]}:=\col(s_i,\ldots,s_j)$ for $0\leq i < j$.
For a set $X\subseteq\Rset^n$ and $N\in\Nset_{\geq 2}$, $X^N$ denotes the $N$-times Cartesian product $X\times\ldots\times X$. The \emph{Kronecker delta} $\delta_{ik}$ for indices
$i, k \in \mathbb{N}$ is defined as
\[  \delta_{ik}
  \;:=\;
  \begin{cases}
    1 & \text{if } i = k, \\
    0 & \text{if } i \neq k.
  \end{cases}
\]
Let $(X, \Sigx, \mux)$ and $(U, \Sigu, \muu)$ be $\sigma$-finite
measure spaces with $X \subseteq \mathbb{R}^n$ and
$U \subseteq \mathbb{R}^m$, and let
$(X \times U,\, \Sigx \otimes \Sigu,\, \mux \otimes \muu)$
denote their product measure space~\cite{Rudin1987}.

\begin{definition}
\label{def:pushfwd}
Let $F : X \times U \to X$ be a measurable map. The
\emph{pushforward} of $\mux \otimes \muu$ under $F$ yields the measure
$\pushfwd(\mux \otimes \muu)$ on $(X, \Sigx)$ defined by
\begin{equation}
  \bigl[\pushfwd(\mux \otimes \muu)\bigr](A)
  \;:=\;
  (\mux \otimes \muu)\bigl(F^{-1}(A)\bigr),
  \quad \forall A \in \Sigx,
\end{equation}
where $F^{-1}(A) := \{(x,u) \in X \times U : F(x,u) \in A\}
\in \Sigx \otimes \Sigu$ is the preimage of $A$ under $F$.
\end{definition}
In this paper we will follow standard notation for function
spaces $L^p(X,\mu_x)$, see, e.g.,~\cite[Chapter~3]{Rudin1987}.
\begin{assumption}
\label{ass:nonsing}
For every $A \in \Sigx$, let $\nu(A): = (\mux \otimes \muu)(F^{-1}(A))$ denote the pushforward measure through $F$. The pushforward measure $\nu$ through $F$ is absolutely continuous with respect to the measure $\mu_x$ (denoted by $\nu \ll \mux$), i.e., it holds that $\mux(A) = 0  \;\implies\;  \nu(A) = 0$.
\end{assumption}
The above assumption states that $F$ does not map sets of positive
$(\mux \otimes \muu)$-measure in $X \times U$ into
$\mux$-null sets in $X$. Then, since $(X, \Sigx, \mux)$ is
$\sigma$-finite and $\nu \ll \mux$, by the
\emph{Radon-Nikodym theorem}~\cite[Theorem~6.10]{Rudin1987} there
exists a unique function $w \in L^1(X, \mux)$, $w \geq 0$
$\mux$-a.e., called the \emph{Radon-Nikodym derivative}
$w := d\nu/d\mux$, such that
\begin{equation}
  \nu(A) \;=\; \int_A w(x')\, d\mux(x'),
  \qquad \forall\, A \in \Sigx.
  \label{eq:RN}
\end{equation}
\begin{assumption}
\label{ass:bounded}
The equality \eqref{eq:RN} holds with a function $w \in L^\infty(X, \mux)$, i.e.,
$M \;:=\; \norm{w}_{L^\infty(X,\mux)} \;<\; \infty$.
\end{assumption}
Assumption~\ref{ass:bounded} requires that $\mux$-a.e., the preimage $F^{-1}(A)$ is bounded relative to $\mux(A)$
uniformly over  measurable $A \in \Sigx$.

We will extensively use $L^2(X,\mu_x)$ and $\ell^2$ spaces, for which we provide detailed definitions in what follows. Let
\begin{equation*}
\begin{split}
    \ell^2 := \{ (x_n)_{n \geq 1} : \sum_{n=1}^{\infty} |x_n|^2 < \infty \},\,
     \|(x_n)\|_{\ell^2} := ( \sum_{n=1}^{\infty} |x_n|^2 )^\frac{1}{2}.
     \end{split}
\end{equation*}
For $X\subseteq\Rset^n$ define $\mathcal{H}_x := L^2(X, \mu_x)$ as
\begin{align*}
  \mathcal{H}_x = \{ \varphi : X \to \mathbb{K}\ \text{measurable}
  \;|\;
  \int_X |\varphi(x)|^2\, d\mu_x(x) < \infty \},
\end{align*}
where functions agreeing $\mu_x$-almost everywhere are
identified, i.e.\ $\varphi_1 = \varphi_2$ in
$\mathcal{H}_x$ iff $\mu_x(\{x : \varphi_1(x) \neq
\varphi_2(x)\}) = 0$.

Furthermore, define the inner product and norm as
\begin{align*}
  \inn{\varphi_1}{\varphi_2}_{\Hx}
  \;&:=\; \int_X \varphi_1(x)\,\overline{\varphi_2(x)}\, d\mux(x),
  \qquad\\
  \norm{\varphi}_{\Hx}
  \;&:=\; \left(\int_X |\varphi(x)|^2\, d\mux(x)\right)^{\!1/2}.
\end{align*}

Let a corresponding Hilbert space on $U\subseteq\Rset^m$, i.e., $\cH_u:=L^2(U,\mu_u)$, be defined similarly. Then we define the corresponding \emph{product Hilbert space} as
\begin{equation}
  \Hprod \;:=\; \Ltwo(X \times U,\, \mux \otimes \muu)
\end{equation}
with inner product and norm
\begin{align}
  \inn{f_1}{f_2}_{\Hprod}
  &\;:=\; \int_{X \times U} f_1(x,u)\,\overline{f_2(x,u)}\,
          d(\mux \otimes \muu), \nonumber\\
  \norm{f}_{\Hprod}
  &\;:=\; \left(\int_{X \times U} |f(x,u)|^2\,
          d(\mux \otimes \muu)\right)^{\!1/2}.\label{eq:3:2:i}
\end{align}

\begin{definition}
\label{def:riesz_basis}
A sequence $\{\varphi_i\}_{i=1}^\infty$ in a Hilbert space
$\mathcal{H}$ is an \emph{orthonormal basis} (ONB) if
$\langle\varphi_i,\varphi_k\rangle_{\mathcal{H}} = \delta_{ik}$
and $\overline{\operatorname{span}}\{\varphi_i\} = \mathcal{H}$.
It is a \emph{Riesz basis} if there exist constants
$0 < \alpha \leq \beta < \infty$ such that for all
$(c_i)_{i\geq 1} \in \ell^2$:
\begin{equation}
  \alpha\sum_{i=1}^\infty|c_i|^2
  \;\leq\;
  \Bigl\|\sum_{i=1}^\infty c_i\,\varphi_i\Bigr\|_{\mathcal{H}}^2
  \;\leq\;
  \beta\sum_{i=1}^\infty|c_i|^2.
  \label{eq:riesz_def}
\end{equation}
\end{definition}
Note that every ONB is a Riesz basis with $\alpha=\beta=1$ and
a Riesz basis is a bounded and boundedly invertible image
of an ONB under a bounded linear
operator~\cite[Theorem~3.6.6]{Christensen2016}.

Furthermore, observe that the tensor product Hilbert space $\Hx\otimes\Hu$ is
isometrically isomorphic to the product space $\Hprod$:
\begin{equation}
  L^2(X,\mux)\otimes L^2(U,\muu)
  \;\cong\;
  L^2(X\times U,\,\mux\otimes\muu),
  \label{eq:tensor_iso}
\end{equation}
via the identification $(\varphi\otimes\psi)(x,u) :=
\varphi(x)\psi(u)$, extended by linearity and
continuity~\cite[Theorem~II.10]{ReedSimon1980}.

The following result follows directly from standard properties of product spaces and Riesz bases \cite[Section~II.4]{ReedSimon1980}.
\begin{proposition}
\label{prop:Hilbert}
If $\{\psi_{i,x}\}_{i=1}^\infty$ is a Riesz
basis for $\Hx$ with bounds $\alpha_x,\beta_x$ and
$\{\psi_{i,u}\}_{i=1}^\infty$ is a Riesz basis for $\Hu$
with bounds $\alpha_u,\beta_u$, then the product functions
\begin{equation}
  \psi_{(i,l)}(x,u) \;:=\; \psi_{i,x}(x)\,\psi_{l,u}(u),
  \qquad i,l\in\mathbb{N},
  \label{eq:product_basis}
\end{equation}
form a Riesz basis for $\Hprod$ with bounds
$\alpha_x\alpha_u \leq \beta_x\beta_u$.
\end{proposition}

\section{The generalized Koopman operator}
\label{sec:3}
For discrete-time autonomous nonlinear systems,
\[x_{t+1}=F_0(x_t), \quad t\in\Nset,\]
with the map $F_0:X\rightarrow X$ Lebesgue integrable, the classical Koopman composition operator \cite{Koopman1931} is defined as
\[C_{F_0} : \Hx \to \Hx,\quad (C_{F_0}\varphi)(x) = \varphi(F_0(x)),\]
where $\varphi\;:\;X\rightarrow\Kset$ is called an observable (function). If $F_0$ is \emph{measure preserving} and bijective, the invariance property \begin{equation}
\label{eq:classic_inv}
\varphi\circ F_0\in\cH_x,\quad\forall\varphi\in\cH_x
\end{equation} holds, which guarantees that the Koopman composition operator is a well-defined bounded linear unitary operator.

In this paper we consider nonlinear dynamical systems with control input, i.e.,
\begin{equation}
\label{eq:3:1}
x_{t+1}=F(x_t,u_t),\quad t\in\Nset,
\end{equation}
where $F:X\times U\rightarrow X$ is a Lebesgue integrable map and $X\subseteq\Rset^n$, $U\subseteq\Rset^m$ denote the state space and the input space.

In what follows we adopt a product Hilbert space approach to define a generalized Koopman composition operator as
\begin{equation}
  C_F : \Hx \to \Hprod,
  \qquad
  (C_F \varphi)(x,u) \;:=\; \varphi(F(x,u)),
  \label{eq:CF}
\end{equation}
where $\cH=\cH_x\otimes\cH_u$ is the product Hilbert space of the state and input Hilbert spaces, respectively. Note that now the generalized Koopman composition operator acts on observable functions $\varphi\in L^2(X,\mu_x)$ and yields functions  $\varphi\circ F\in L^2(X\times U,\mu_x\otimes\mu_u)$. As such, this yields a novel, \emph{generalized invariance} condition, i.e.,\begin{equation}
\label{eq:gen_inv}
\varphi\circ F \in \cH=\cH_x\otimes\cH_u, \quad \forall \varphi\in\cH_x.
\end{equation}

Conditions on the map $F$ that yield a well-defined generalized Koopman composition operator are given next.

\begin{theorem}
\label{thm:bounded_CF}
Let $(X, \Sigx, \mux)$ and $(U, \Sigu, \muu)$ be $\sigma$-finite
measure spaces and let $F : X \times U \to X$ be measurable. Define
the pushforward $\nu := \pushfwd(\mux \otimes \muu)$ on $(X, \Sigx)$
as in Definition~\ref{def:pushfwd} and suppose Assumptions~\ref{ass:nonsing} and~\ref{ass:bounded} hold. Then the
\emph{generalized Koopman composition operator}
\begin{equation*}
  C_F : \Hx \to \Hprod,
  \qquad
  (C_F \varphi)(x,u) \;=\; \varphi(F(x,u)),
  \end{equation*}
is a well-defined \emph{bounded linear operator} with operator norm
\begin{equation*}
  \norm{C_F}_{\Hx\to\mathcal{H}}
  \;:=\; \sup_{\substack{\varphi \in \Hx \\ \norm{\varphi}_{\Hx}=1}}
         \norm{C_F \varphi}_{\Hprod}
  \;\leq\; M^{1/2}.
  \end{equation*}
\end{theorem}
\begin{proof}
\emph{Linearity:} For $\varphi_1, \varphi_2 \in \Hx$ and
$\alpha, \beta \in \mathbb{K}$:
\begin{align*}
  (C_F(\alpha\varphi_1 &+ \beta\varphi_2))(x,u)
  = (\alpha\varphi_1 + \beta\varphi_2)(F(x,u)) \\
  &= \alpha\varphi_1(F(x,u)) + \beta\varphi_2(F(x,u)),\\
  &=\alpha (C_F\varphi_1)(x,u) + \beta (C_F\varphi_2)(x,u).
\end{align*}
\emph{Boundedness:} Since $F$ is measurable and $\varphi$ is
measurable, the composition $\varphi \circ F$ is measurable on
$(X \times U, \Sigx \otimes \Sigu)$. For any $\varphi \in \Hx$, one can then apply the
\emph{change-of-variables formula} for pushforward
measures~\cite[eq. (1.1)]{Villani2009}, i.e., if $T:(Y,\mathcal{B})
\to (Z,\mathcal{C})$ is measurable and $\lambda$ is a measure on
$\mathcal{B}$, then
$\int_Z f\, d(T_*\lambda) = \int_Y f \circ T\, d\lambda$
for any non-negative measurable $f$. Here $T_*\lambda$ denotes the \emph{pushforward} of the measure $\lambda$ under the map $T$, defined by
$(T_*\lambda)(A) := \lambda(T^{-1}(A))$ for measurable
sets $A$ (cf.\ Definition~\ref{def:pushfwd}). In our setting, $Y = X \times U$ with $\lambda = \mux \otimes \muu$, and $Z = X$ with $T_*\lambda = \nu$ the pushforward measure on $X$. Then the change-of-variables formula with $T = F$ and
$f = |\varphi|^2 \geq 0$ yields
\begin{align*}
  \norm{C_F \varphi}_{\cH}^2
  &= \int_{X \times U} |(\varphi \circ F)(x,u)|^2\,
    d(\mux \otimes \muu)
  \nonumber\\
  &= \int_{X} |\varphi(x')|^2\, d\nu(x').
  \end{align*}
Substituting $d\nu = w\, d\mux$ from~\eqref{eq:RN} further yields
\begin{equation*}
  \norm{C_F \varphi}_{\Hprod}^2= \int_X |\varphi(x')|^2\, w(x')\, d\mux(x').
\end{equation*}
Applying Assumption~\ref{ass:bounded}, $w(x') \leq M$
for $\mux$-a.e.\ $x'$ gives
\begin{equation*}
 \norm{C_F \varphi}_{\Hprod}^2 \leq M \int_X |\varphi(x')|^2\, d\mux(x')
  \;=\; M \norm{\varphi}_{\Hx}^2 \;<\; \infty.
\end{equation*}
Taking the supremum over $\norm{\varphi}_{\Hx} = 1$ yields
$\norm{C_F}_{\Hx\to\mathcal{H}} \leq M^{1/2}$.
\end{proof}

Theorem~\ref{thm:bounded_CF} establishes that under
Assumptions~\ref{ass:nonsing} and~\ref{ass:bounded},
$C_F: \Hx \to \Hprod$ is bounded with
$\|C_F\| \leq M^{1/2}$, which implies
$\varphi \circ F \in \Hprod$ for every
$\varphi \in \Hx$ since
$\|\varphi\circ F\|_{\Hprod}^2
= \|C_F\varphi\|_{\Hprod}^2
\leq M\|\varphi\|_{\Hx}^2 < \infty$.
This yields the generalized invariance
condition~\eqref{eq:gen_inv}. A few remarks are in order.
\begin{remark}[Relation with the Koopman operator]
The classical~\cite{Koopman1931} \emph{Koopman invariance condition}~\eqref{eq:classic_inv}
requires $\mux$ to be \emph{$F_0$-invariant}, i.e.,
$ \mux(A) = \mux(F_0^{-1}(A))$, for all $A \in \Sigx$, which is exactly $(F_0)_*\mux = \mux$. This further implies
\begin{equation*}
  w_0 \;:=\; \frac{d\,(F_0)_*\mux}{d\mux} \;\equiv\; 1
  \quad \mux\text{-a.e.}
\end{equation*}
Then by the change-of-variables formula
and $(F_0)_*\mux = \mux$:
\begin{align*}
  \norm{C_{F_0}\varphi}_{\Hx}^2
  &= \int_X |(\varphi\circ F_0)(x)|^2\,d\mux(x)  \\
  &= \int_X |\varphi(x')|^2\,d[(F_0)_*\mux](x')\\
  &= \int_X |\varphi(x')|^2\,d\mux(x') \;=\; \norm{\varphi}_{\Hx}^2.
\end{align*}
Thus $C_{F_0}$ is an \emph{isometry} with $\norm{C_{F_0}}_{\Hx\to\Hx} = 1$.
If additionally $F_0$ is bijective and measure-preserving, then $(F_0^{-1})_*\mux = \mux$ holds automatically and
$C_{F_0}$ is \emph{unitary}, recovering the classical setting of~\cite{Koopman1931}, which was also employed in~\cite{Korda_2018_convergence, Korda_2018_Koop}. Theorem~\ref{thm:bounded_CF} generalizes this by allowing for $X\times U$ as the domain of $F$ and deriving a \emph{non-measure preserving} bounded composition operator.
The steps taken in relaxing the operator conditions  are summarized in the following  in
Table~\ref{tab:comparison}, where $\nu_0 := (F_0)_*\mux$ and
$w_0 := d\nu_0/d\mux$.
\begin{table}[h]
\renewcommand{\arraystretch}{1.4}
\caption{Comparison of Koopman Operator Settings}
\label{tab:comparison}
\centering
\begin{tabular}{p{2.6cm} p{1.7cm} p{2.5cm}}
\toprule
\textbf{Condition on $F$} & \textbf{R.-N.\ deriv.} & \textbf{Operator property} \\
\midrule
$F_0\!:\!X\!\to\! X$, meas.-pres., bijective
  & $w_0 \!\equiv\! 1$
  & $C_{F_0}\!:\!\Hx\!\to\!\Hx$ unitary, $\norm{C_{F_0}}\!=\!1$ \\
$F_0\!:\!X\!\to\! X$, meas.-pres., non-bijective
  & $w_0 \!\equiv\! 1$
  & $C_{F_0}\!:\!\Hx\!\to\!\Hx$ isometry, $\norm{C_{F_0}}\!=\!1$ \\
$F_0\!:\!X\!\to\! X$, non-meas.-pres., $w_0\!\in\! L^\infty$
  & $0\!\leq\! w_0\!\leq\! M$
  & $C_{F_0}\!:\!\Hx\!\to\!\Hx$ bounded, $\norm{C_{F_0}}\!\leq\! M^{1/2}$ \\
$F\!:\!X\!\times\! U\!\to\! X$, $w\!\in\! L^\infty$
  & $0\!\leq\! w\!\leq\! M$
  & $C_F\!:\!\Hx\!\to\!\Hprod$ bounded, $\norm{C_F}\!\leq\! M^{1/2}$ \\
\bottomrule
\end{tabular}
\end{table}
\end{remark}
Note that the use of weighted $L^p$ spaces to establish
boundedness of the composition operator for autonomous systems via
the Radon-Nikodym derivative was previously employed in~\cite{BreitenHoveler2023},
in the context of operator Lyapunov equations.
\begin{remark}[Implications for $F$]
\label{rem:classes}
Assumptions~\ref{ass:nonsing}--\ref{ass:bounded} hold whenever
$J_F(x,u) := |\det(\partial F/\partial x)(x,u)| \geq c > 0$
a.e.\ on $X\times U$ for some constant $c>0$, so that
$F_u := x\mapsto F(x,u)$ is a diffeomorphism on $X$ for each
fixed $u\in U$, with inverse $F_u^{-1}$. Taking $\mu_x$ as the
Lebesgue measure on $X\subseteq\Rset^n$ and applying Fubini's
theorem and the diffeomorphism change-of-variables
formula~\cite[Thm.~7.26]{Rudin1987} gives
\begin{align*}
  \nu(A)
  &= \int_{U}\mu_x\bigl(\{x: F_u(x)\in A\}\bigr)\,d\mu_u(u)\\
  &= \int_{U}\int_A \frac{d\mu_x(x')}{J_F(F_u^{-1}(x'),u)}\,d\mu_u(u)  \leq \frac{\mu_u(U)}{c}\,\mu_x(A),
\end{align*}
for all $A\in\Sigma_x$ ($\mu_x(A) := \int_A d\mu_x$ and similarly for $\mu_u(U)$), where the first equality applies Fubini, the second applies the
change of variables $x' = F_u(x)$, and the inequality uses
$J_F \geq c$. It follows immediately that $\mu_x(A)=0 \Rightarrow \nu(A)=0$,
i.e.\ $\nu\ll\mu_x$ (Assumption~\ref{ass:nonsing}).
By the Radon--Nikodym theorem~\cite[Theorem~6.10]{Rudin1987},
there exists $w = d\nu/d\mu_x \in L^1(X,\mu_x)$ such
that $\nu(A) = \int_A w\,d\mu_x$ for all $A\in\Sigma_x$, and
the above inequality implies $w(x') \leq c^{-1}\mu_u(U)$
$\mu_x$-a.e., hence
$M := \|w\|_{L^\infty} \leq c^{-1}\mu_u(U) < \infty$
(Assumption~\ref{ass:bounded}). The condition $J_F \geq c > 0$ is verifiable. For example, for an Euler discretisation, i.e., $F(x,u): = x + T_sf(x,u)$ with $f(x,u)$ the continuous-time system map, one has
$J_F = |\det(I + T_s\,\partial f/\partial x)|$,
and if $\|\partial f/\partial x\|_{\mathrm{op}} \leq L$
uniformly on $X\times U$ with $T_s < 1/L$, then by the
bound $|\det(I+T_sQ)|\geq(1-T_s\|Q\|_{\mathrm{op}})^n$
one obtains $J_F \geq (1-LT_s)^n =: c > 0$. Here $\|\cdot\|_{\mathrm{op}}$ denotes
the matrix operator norm (largest singular value). In the classical autonomous case, $F_0$ is measure-preserving,
which corresponds to $J_F \equiv 1$, a strictly stronger condition
than $J_F \geq c > 0$. The latter permits bounded volume distortion
but excludes volume collapse, and is sufficient to guarantee that
$C_F$ is a bounded operator as shown in
Theorem~\ref{thm:bounded_CF}.
\end{remark}
\subsection{Coordinate representation in Riesz bases}
Having established that $C_F : \Hx \to \mathcal{H}$ is a
bounded operator under Assumptions~\ref{ass:nonsing}
--\ref{ass:bounded}, next we will construct its coordinate representation in an orthonormal/Riesz basis as an operator
$\cK : \ell^2 \to \ell^2$. This will facilitate the derivation of an exact infinite-dimensional bilinear Koopman model for general nonlinear maps $F$. To this aim, consider the sets of $\Kset$-valued observable functions  $\{\psi_{i,x}\}_{i=1}^\infty$, $\{\psi_{j,u}\}_{j=1}^\infty$  with $\psi_{i,x}\in\cH_x$, $\psi_{j,u}\in\cH_u$ for all $i,j\in\Nset$. Furthermore, we construct lifted state and input vectors, i.e., $\Psi_x(x_t)$, and $\Psi_u(u_t)$, by stacking the corresponding observable functions, e.g., $\Psi_x(x_t)=\col(\psi_{1,x}(x_t),\psi_{2,x}(x_t),\ldots )$.

Next, define $\Psi(x,u):=\Psi_x(x)\otimes\Psi_u(u)$ with $\Psi(x,u)=\col(\psi_1(x,u),\psi_2(x,u),\ldots)$ and
\begin{equation}
\label{eq:3:basis}
\{\psi_1,\psi_2,\ldots\}:=\{\psi_{1,x}\psi_{1,u},\psi_{1,x}\psi_{2,u},\ldots,\psi_{i,x}\psi_{j,u},\ldots \},
\end{equation}
i.e., $\{\psi_l\}_{l=1}^\infty$ contains products $\{\psi_{i,x}\psi_{j,u}\}$, $i,j\in\Nset$, which is a basis of the tensor product $\cH_x\otimes\cH_u$ cf. Propostion~\ref{prop:Hilbert}.
\begin{proposition}
\label{prop:1}
Let $\{\varphi_{i}\}_{i=1}^\infty$ and $\{\psi_{i,x}\}_{i=1}^\infty$ be orthonormal bases in $\cH_x$ and let $\{\psi_{i,u}\}_{i=1}^\infty$ be an orthonormal basis in $\cH_u$. Let the state transition map $F : X\times U\rightarrow X$ in \eqref{eq:3:1} satisfy:
\begin{equation}
\label{eq:3:prop1}
\varphi_{i}\circ F\in\cH=\cH_x\otimes\cH_u,\quad i\in\Nset.
\end{equation}
Then there exists an infinite-dimensional matrix $\cK$ such that
\begin{equation}
\label{eq:3:3}
\Psi_x[F(x,u)]=\cK\Psi(x,u)=\cK\left(\Psi_x(x)\otimes\Psi_u(u)\right),
\end{equation}
with $\Psi_x[F(x,u)]=\col(\psi_{1,x}[F(x,u)],\psi_{2,x}[F(x,u)],\ldots)$ and
\begin{equation*}
\begin{split}
&\cK=\\
&\sum_{j=1}^\infty \begin{bmatrix}\langle\psi_{1,x},\varphi_{j}\rangle_{\cH_x}\\\langle\psi_{2,x},\varphi_{j}\rangle_{\cH_x}\\\vdots\end{bmatrix}\begin{bmatrix}\langle\varphi_{j}\circ F,\psi_1\rangle_{\cH} & \langle\varphi_{j}\circ F,\psi_2\rangle_{\cH}& \ldots\end{bmatrix}.
\end{split}
\end{equation*}
\end{proposition}
\begin{proof}
Since each $\psi_{i,x}\in\cH_x$, $i\in\Nset$, we can expand each $\psi_{i,x}$ as $\psi_{i,x}=\sum_{j=1}^\infty\langle\psi_{i,x},\varphi_j\rangle_{\cH_x}\varphi_j$. Then we have that
\begin{equation}
\label{eq:3:4}
\begin{split}
\Psi_x\circ F=\begin{bmatrix}\psi_{1,x}\circ F\\\psi_{2,x}\circ F\\\vdots\end{bmatrix}&=\begin{bmatrix}\sum_{j=1}^\infty\langle\psi_{1,x},\varphi_j\rangle_{\cH_x}\varphi_j\circ F\\\sum_{j=1}^\infty\langle\psi_{2,x},\varphi_j\rangle_{\cH_x}\varphi_j\circ F\\\vdots\end{bmatrix}\\
&=\sum_{j=1}^\infty\begin{bmatrix}\langle\psi_{1,x},\varphi_j\rangle_{\cH_x}\\\langle\psi_{2,x},\varphi_j\rangle_{\cH_x}\\\vdots\end{bmatrix}\varphi_j\circ F.
\end{split}
\end{equation}
The second equality holds row-by-row, i.e., for each fixed $i\in\Nset$,
expanding $\psi_{i,x}$ in the ONB $\{\varphi_j\}_{j=1}^\infty$
and passing $C_F$ through the sum is justified by boundedness of
$C_F$ (Theorem~\ref{thm:bounded_CF}), giving
$\psi_{i,x}\circ F = \sum_{j=1}^\infty
\langle\psi_{i,x},\varphi_j\rangle_{\cH_x}\varphi_j\circ F$
converging in $\mathcal{H}$.
The last equality is a purely algebraic rearrangement, i.e.,
since each row already converges in $\mathcal{H}$, the summation
and vector stacking can be interchanged, grouping terms by $j$
rather than by $i$.

Next, by Proposition~\ref{prop:Hilbert} we have that $\{\psi_i\}_{i=1}^\infty$ (with $\psi_i$ as defined in \eqref{eq:3:basis}) is an orthonormal basis in $\cH=\cH_x\otimes\cH_u$. Then, since $\varphi_j\circ F\in\cH=\cH_x\otimes\cH_u$ for $j\in\Nset$ it holds that
\begin{equation}
\label{eq:3:5}
\varphi_j\circ F = \sum_{i=1}^\infty\langle\varphi_j\circ F,\psi_i\rangle_{\cH}\psi_i,
\end{equation}
where the inner product on $\cH$ is defined as in \eqref{eq:3:2:i}. Substituting \eqref{eq:3:5} in \eqref{eq:3:4} yields
\[
\begin{split}
&\Psi_x[F(x,u)]=\sum_{j=1}^\infty\begin{bmatrix}\langle\psi_{1,x},\varphi_j\rangle_{\cH_x}\\\langle\psi_{2,x},\varphi_j\rangle_{\cH_x}\\\vdots\end{bmatrix}\varphi_j\circ F(x,u)\\
&=\sum_{j=1}^\infty\begin{bmatrix}\langle\psi_{1,x},\varphi_j\rangle_{\cH_x}\\\langle\psi_{2,x},\varphi_j\rangle_{\cH_x}\\\vdots\end{bmatrix}\sum_{i=1}^\infty\langle\varphi_j\circ F,\psi_i\rangle_{\cH}\psi_i(x,u)=\\
&\sum_{j=1}^\infty\begin{bmatrix}\langle\psi_{1,x},\varphi_j\rangle_{\cH_x}\\\langle\psi_{2,x},\varphi_j\rangle_{\cH_x}\\\vdots\end{bmatrix}[\langle\varphi_j\circ F, \psi_1\rangle_{\cH}, \langle\varphi_j\circ F, \psi_2\rangle_{\cH}, \ldots]\cdot\\
&\cdot \begin{bmatrix}\psi_1(x,u)\\\psi_2(x,u)\\\vdots\end{bmatrix}=\cK\Psi(x,u),
\end{split}
\]
which completes the proof.
\end{proof}

The above result generalizes \cite[Proposition~2]{Asada} to state transition maps $F(x,u)$ with control inputs under the generalized invariance property \eqref{eq:3:prop1}; sufficient conditions on $F$ that guarantee this property have already been identified in Theorem~\ref{thm:bounded_CF}. By letting $z_t:=\Psi_x(x_t)$ and $v_t:=\Psi_u(u_t)$, we can define a Koopman infinite-dimensional bilinear system:
\begin{equation}
\label{eq:3:2}
\begin{split}
z_{t+1}&=\Psi_x(x_{t+1})=\Psi_x[F(x_t,u_t)]=\cK \left(\Psi_x(x_t)\otimes\Psi_u(u_t)\right)\\
&=\cK (z_t\otimes v_t),
\end{split}
\end{equation}
where $\{x_t,u_t\}_{t\in\Nset}$ is a solution of system \eqref{eq:3:1}. The developed construction of the bilinear Koopman model for systems with inputs is graphically illustrated in Figure~\ref{fig2}.
\begin{figure}[!t]
\centerline{\includegraphics[width=0.7\columnwidth]{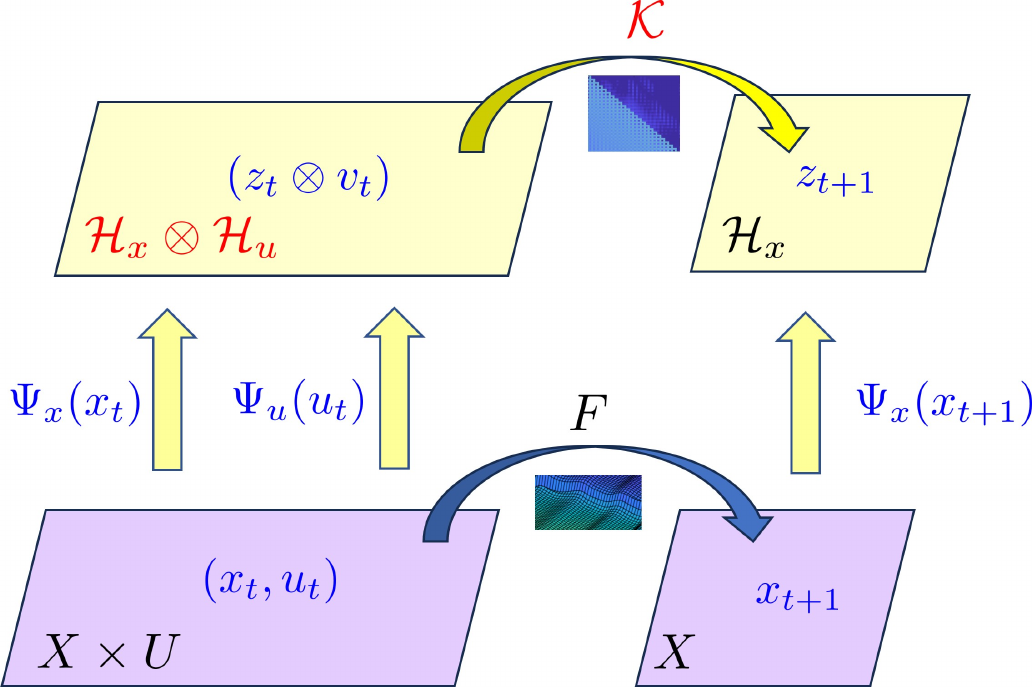}}
\caption{Illustration of the envisioned generalized Koopman operator architecture for systems with control input.}
\label{fig2}
\end{figure}

As shown in \cite[Corollary~1]{Asada} for the case of autonomous state transition maps, if in Proposition~\ref{prop:1} we use $\psi_{i,x}=\varphi_{i}$, $i\in\Nset$, the infinite-dimensional matrix $\cK$ can be simplified to
\begin{equation}
\label{eq:3:6}
\bar \cK=\begin{bmatrix}\langle\varphi_1\circ F,\psi_1\rangle_{\cH}&\langle\varphi_1\circ F,\psi_2\rangle_{\cH}&\ldots\\
\langle\varphi_2\circ F,\psi_1\rangle_{\cH}&\langle\varphi_2\circ F,\psi_2\rangle_{\cH}&\ldots\\
\vdots&\vdots&\ddots\end{bmatrix},
\end{equation}
with $\psi_i(x,u)=\varphi_j(x)\psi_{l,u}(u)$ for all $i\in\Nset$ for some $j,l\in\Nset$.

Next, we show that the condition on the sets of observable functions $\{\psi_{i,x}\}_{i=1}^\infty$, $\{\psi_{i,u}\}_{i=1}^\infty$ can be relaxed to a linearly independent and complete set of functions that form a Riesz basis in $\cH_x$ and $\cH_u$, respectively, instead of an orthonormal basis. First, notice that we can expand $\psi_{i,x}=\sum_{j=1}^\infty\langle\psi_{i,x},\varphi_j\rangle\varphi_j$, which can be compactly written as
\begin{equation}
\label{eq:3:cx}
\begin{bmatrix}\psi_{1,x}\\\psi_{2,x}\\\vdots\end{bmatrix}=
\underbrace{\begin{bmatrix}\langle\psi_{1,x},\varphi_1\rangle_{\cH_x}&\langle\psi_{1,x},\varphi_2\rangle_{\cH_x}&\ldots\\\langle\psi_{2,x},\varphi_1\rangle_{\cH_x}&\langle\psi_{2,x},\varphi_2\rangle_{\cH_x}&\ldots\\\vdots&\vdots&\ddots\end{bmatrix}}_{:=T_x}\begin{bmatrix}\varphi_1\\\varphi_2\\\vdots\end{bmatrix}.
\end{equation}
The same transformation can be done for the product space $\cH$, i.e., $\psi_i(x,u)=\sum_{j=1}^\infty\langle\psi_i,e_j\rangle e_j(x,u)$, which can be compactly written as
\begin{equation}
\label{eq:3:c}
\begin{bmatrix}\psi_1\\\psi_2\\\vdots\end{bmatrix}=
\underbrace{\begin{bmatrix}\langle\psi_1,e_1\rangle_{\cH}&\langle\psi_1,e_2\rangle_{\cH}&\ldots\\\langle\psi_2,e_1\rangle_{\cH}&\langle\psi_2,e_2\rangle_{\cH}&\ldots\\\vdots&\vdots&\ddots\end{bmatrix}}_{:=T}\begin{bmatrix}e_1\\ e_2\\\vdots\end{bmatrix},
\end{equation}
where $\{e_i\}_{i=1}^\infty$ is a generic ONB in $\cH$. Note that these transformations amount to a change of Riesz basis, since orthonormal bases are Riesz bases \cite[Chapter~3.6]{Christensen2016}, and thus, the infinite-dimensional matrices $T_x, T$ are nonsingular, i.e., they are invertible and bounded operators \cite[Theorem~3.6.6]{Christensen2016}.

\begin{corollary}
\label{cor:1}
If the observables $\{\psi_{i,x}\}_{i=1}^\infty$ and $\{\psi_{i,u}\}_{i=1}^\infty$ form a Riesz basis in $\cH_x$ and $\cH_u$, respectively, then the infinite-dimensional matrix $\cK$ in \eqref{eq:3:3} is given by
\begin{equation}
\label{eq:3:7}
\cK=T_x\bar \cK T^{-1},
\end{equation}
with $\bar\cK$, $T_x$ and $T$ defined as in \eqref{eq:3:6}-\eqref{eq:3:c}, respectively.
\end{corollary}
So far, we have shown the existence of the matrix $\cK$, but its expression depends on a generic, ONB $\{\varphi_i\}_{i=1}^\infty$ in $\cH_x$. Next, we prove that $\cK$ can be constructed just based on the state and input observable functions and the map $F$.
\begin{theorem}
\label{thm:1}
Let $\{\psi_{i,x}\}_{i=1}^\infty$ and $\{\psi_{i,u}\}_{i=1}^\infty$ be independent and complete sets of Riesz basis functions spanning $\cH_x$, and $\cH_u$, respectively. Let the state transition map $F : X\times U\rightarrow X$ in \eqref{eq:3:1} satisfy:
\begin{equation}
\label{eq:3:8}
\psi_{i,x}\circ F\in\cH,\quad i\in\Nset.
\end{equation}
Then there exists an infinite-dimensional matrix $\cK=Q_xR^{-1}$  such that
\begin{equation}
\label{eq:3:9}
\Psi_x(F(x,u))=\cK\Psi(x,u)=\cK(\Psi_x(x)\otimes\Psi_u(u)),\end{equation}
with the matrices $R$ and $Q_x$ given by:
\begin{equation}
\label{eq:3:10}
R=\begin{bmatrix}\langle\psi_1,\psi_1\rangle_{\cH} &\langle\psi_1,\psi_2\rangle_{\cH}&\ldots\\\langle\psi_2,\psi_1\rangle_{\cH}&\langle\psi_2,\psi_2\rangle_{\cH}&\ldots\\\vdots&\vdots&\ddots\end{bmatrix}
\end{equation}
and
\begin{equation}
\label{eq:3:11}
Q_x=\begin{bmatrix}\langle\psi_{1,x}\circ F,\psi_1\rangle_{\cH} &\langle\psi_{1,x}\circ F,\psi_2\rangle_{\cH}&\ldots\\\langle\psi_{2,x}\circ F,\psi_1\rangle_{\cH}&\langle\psi_{2,x}\circ F,\psi_2\rangle_{\cH}&\ldots\\\vdots&\vdots&\ddots\end{bmatrix}.
\end{equation}
\end{theorem}
\begin{proof}
From \eqref{eq:3:cx} we obtain \[\begin{bmatrix}\varphi_1\circ F\\\varphi_2\circ F\\\vdots\end{bmatrix}=T_x^{-1}\begin{bmatrix}\psi_{1,x}\circ F\\\psi_{2,x}\circ F\\\vdots\end{bmatrix},\]
which, together with \eqref{eq:3:8} implies $\varphi_i\circ F\in\cH$ for all $i\in\Nset$. Then, from Proposition~\ref{prop:1} we have that there exists an infinite-dimensional matrix $\cK$ such that \eqref{eq:3:9} holds. Multiplying both sides of the
equation~\eqref{eq:3:9} on the right by the row
vector $[\psi_1,\psi_2,\ldots]$ gives the matrix
of functions equality
\begin{equation}
  \bar{Q}(x,u) \;=\; \cK\bar{R}(x,u)
    \label{eq:Qbar_Rbar}
\end{equation}
where $\{\psi_j\}_{j=1}^\infty$ is the product
Riesz basis of $\cH = \Hx\otimes\Hu$ and
\begin{align*}
  [\bar{Q}(x,u)]_{ij}
  \;&:=\; \psi_{i,x}(F(x,u))\cdot\psi_j(x,u),\\
  [\bar{R}(x,u)]_{ij}
  \;&:=\; \psi_i(x,u)\cdot\psi_j(x,u).
\end{align*}
For each fixed $i,j\in\mathbb{N}$, integrating
the $(i,j)$-th entry of both sides
of~\eqref{eq:Qbar_Rbar} over $X\times U$ gives
on the left:
\begin{equation*}
\begin{split}
  &\int_{X\times U}[\bar{Q}(x,u)]_{ij}\,
  d(\mu_x\otimes\mu_u)\\
  \;&=\;
  \int_{X\times U}\psi_{i,x}(F(x,u))\psi_j(x,u)\,
  d(\mu_x\otimes\mu_u)\\
  \;&=\;
  \langle\psi_{i,x}\circ F,\,\psi_j\rangle_{\cH}
  \;=\; [Q_x]_{ij},
  \end{split}
\end{equation*}
and on the right:
\begin{equation*}
\begin{split}
  &\int_{X\times U}[\cK\bar{R}(x,u)]_{ij}\,
  d(\mu_x\otimes\mu_u)\\
  \;&=\;
  \int_{X\times U}\sum_{k=1}^\infty
  [\cK]_{ik}\,\psi_k(x,u)\psi_j(x,u)\,
  d(\mu_x\otimes\mu_u).
  \end{split}
\end{equation*}
Define the sums
$f_N(x,u) := \sum_{k=1}^N [\cK]_{ik}
\psi_k(x,u)\psi_j(x,u)$,
so that the infinite sum is the pointwise limit
$\lim_{N\to\infty}f_N(x,u) =
\sum_{k=1}^\infty[\cK]_{ik}\psi_k(x,u)\psi_j(x,u)$.
Next we establish that
\begin{equation}
  \int_{X\times U}\lim_{N\to\infty}f_N\,
  d(\mu_x\otimes\mu_u)
  \;=\;
  \lim_{N\to\infty}\int_{X\times U}f_N\,
  d(\mu_x\otimes\mu_u).
  \label{eq:swap}
\end{equation}
This holds by the dominated convergence
theorem~\cite[Theorem~1.34]{Rudin1987} if there
exists $g\in L^1(X\times U,\mu_x\otimes\mu_u)$
such that $|f_N(x,u)|\leq g(x,u)$ for all $N$.
Such a dominating function exists as follows: Cauchy-Schwarz pointwise gives
\begin{align*}
  |f_N(x,u)|
  &\leq
  \Bigl(\sum_{k=1}^N|[\cK]_{ik}|^2\Bigr)^{1/2}
  \Bigl(\sum_{k=1}^N|\psi_k(x,u)|^2\Bigr)^{1/2}
  |\psi_j(x,u)|
  \nonumber\\
  &\leq
  \|\cK\|
  \Bigl(\sum_{k=1}^\infty|\psi_k(x,u)|^2\Bigr)^{1/2}
  |\psi_j(x,u)|
  =: g(x,u),
\end{align*}
uniformly in $N$, where the second inequality bounds the finite
row norm by $\|\cK\|$ and extends the sum to infinity.
Then the Bessel upper bound gives
$\sum_k|\psi_k(x,u)|^2\leq\beta$ a.e., so
$g(x,u) \leq \|\cK\|\,\beta^{1/2}\,|\psi_j(x,u)|$
and hence
$\int g\,d(\mux\otimes\muu)
\leq\|\cK\|\,\beta^{1/2}\,(\mu_x(X)\mu_u(U))^{1/2}\|\psi_j\|_{\cH}<\infty$,
where we used
\begin{align*}
\int_{X\times U} |\psi_j|\, d(\mux\otimes\muu) &\leq \|1\|_{\cH}\|\psi_j\|_{\cH} \\&= \bigl(\mu_x(X)\,\mu_u(U)\bigr)^{1/2}\|\psi_j\|_{\cH} < \infty.
\end{align*}
Thus $g\in L^1$ and~\eqref{eq:swap} holds by the
dominated convergence theorem~\cite[Theorem~1.34]{Rudin1987}.
Hence, it follows that
\begin{equation*}
\begin{split}
  &\int_{X\times U}[\cK\bar{R}(x,u)]_{ij}\,
  d(\mu_x\otimes\mu_u)\\\;&=\;\int_{X\times U}\sum_{k=1}^\infty
  [\cK]_{ik}\psi_k(x,u)\psi_j(x,u)\,
  d(\mu_x\otimes\mu_u)\\
  \;&=\;
  \sum_{k=1}^\infty [\cK]_{ik}
  \underbrace{
  \int_{X\times U}\psi_k(x,u)\psi_j(x,u)\,
  d(\mu_x\otimes\mu_u)
  }_{=\,\langle\psi_k,\psi_j\rangle_{\cH}
  \,=\,[R]_{kj}}
  \;=\; [\cK R]_{ij}.
  \end{split}
\end{equation*}
Since $[Q_x]_{ij} = [\cK R]_{ij}$ holds for all
$i,j\in\mathbb{N}$, we conclude $Q_x = \cK R$. Since $\{\psi_i\}_{i=1}^\infty$ is a Riesz basis in $\cH=\cH_x\otimes\cH_u$, its Gram matrix $R$ is nonsingular/invertible~\cite[Theorem~3.6.6]{Christensen2016}, which yields $\cK=Q_xR^{-1}$.
\end{proof}

\begin{remark}[Boundedness of $\cK$ and its relation to $C_F$]
\label{rem:K_bounded}
The infinite-dimensional matrix $\cK = Q_xR^{-1}$
defines a bounded linear operator $\cK:\ell^2\to\ell^2$. Indeed,
\begin{equation}
  \cK \;=\; Q_xR^{-1}
  \;=\; S^*\circ C_F\circ S_x\circ R^{-1},
  \label{eq:K_factorization}
\end{equation}
where $S_x:\ell^2\to\Hx$ and $S:\ell^2\to\cH$
are the synthesis operators~\cite[Ch.~3]{Christensen2016} of the Riesz
bases in $\cH_x$ and $\cH$, respectively, with adjoints
$S_x^*:\Hx\to\ell^2$ and $S^*:\cH\to\ell^2$ the corresponding
analysis operators. Above, $R$ is the
Gram matrix of the product Riesz basis in $\cH$ with $[R]_{i,j}
= \langle\psi_i(x,u),\psi_{j}(x,u)
\rangle_{\mathcal{H}}$.
The lifted dynamics $z_{t+1} = \cK(z_t\otimes v_t)$
are the coordinate-space expression of
$\psi_{i,x}(x_{t+1}) = \psi_{i,x}(F(x_t,u_t))$
for all $i\in\mathbb{N}$, connected to $C_F$ via
the Riesz basis isomorphism. Boundedness of $\mathcal{K}$ follows from the
factorization $\mathcal{K} = Q_x R^{-1}$
in~\eqref{eq:K_factorization}.
The factor $R^{-1}:\ell^2\to\ell^2$ is bounded
with $\|R^{-1}\|\leq\alpha^{-1}$, since the
Riesz basis lower bound gives $R\geq\alpha I$
on $\ell^2$~\cite[Theorem~3.6.6]{Christensen2016}.
 $Q_x:\ell^2\to\ell^2$ is bounded since
\begin{equation}
  \|Q_x\|
  \leq \|S^*\|\,\|C_F\|_{\mathcal{H}_x\to\mathcal{H}}\,\|S_x\|
  \leq \beta\, M^{1/2},
\end{equation}
where $\|S_x\|,\|S^*\|\leq\beta^{1/2}$ by the
Bessel property of the Riesz
bases~\cite[Theorem~3.2.3]{Christensen2016}
and $\|C_F\|_{\mathcal{H}_x\to\mathcal{H}}\leq M^{1/2}$
by Theorem~\ref{thm:bounded_CF}.
Submultiplicativity then gives
\begin{equation}
  \|\mathcal{K}\|_{\ell^2\to\ell^2}
  \;\leq\;
  \frac{M^{1/2}\beta}{\alpha},
  \label{eq:K_norm_bound}
\end{equation}
where $\beta/\alpha$ is the Riesz basis
condition number.
\end{remark}

\subsection{Input-affine nonlinear systems}
Consider the input-affine discrete-time system
\begin{equation}
  x_{t+1} \;=\; F(x_t, u_t)
  \;:=\; f(x_t) + \sum_{j=1}^{m} g_j(x_t)\, u_{j,t},
  \label{eq:inputaffine}
\end{equation}
with $f, g_j : X \to \mathbb{R}^n$ measurable,
$u = (u_1,\ldots,u_m)^\top \in U\subseteq\Rset^m $. 
\begin{definition}
\label{def:psi_u}
Let the $m+1$ scalar-valued input observables
$\psi_{j,u} : U \to \mathbb{R}$ be defined as
\begin{equation}
  \psi_{0,u}(u) \;:=\; 1,
  \qquad
  \psi_{j,u}(u) \;:=\; u_j,
  \quad j = 1, \ldots, m.
  \label{eq:psi_u_def}
\end{equation}
The corresponding input observable subspace is
\begin{equation}
  \mathcal{V}_u
  \;:=\; \operatorname{span}\{\psi_{0,u}, \psi_{1,u},
         \ldots, \psi_{m,u}\}
  \;\subset\; L^2(U,\mu_u),
  \label{eq:Vu}
\end{equation}
which has dimension $m+1$ since $\{\psi_{j,u}\}_{j=0}^m$
are linearly independent in $L^2(U,\mu_u)$ for any $U\subseteq\mathbb{R}^m$
of positive Lebesgue measure.
Define $\mathcal{H}_u := \mathcal{V}_u$ and the map
\begin{equation}
\begin{split}
  \Psi_u(u_t) \;:=\; v_t
  \;&:=\; \bigl(\psi_{0,u}(u_t),\, \psi_{1,u}(u_t),\,
               \ldots,\, \psi_{m,u}(u_t)\bigr)^\top\\
  \;&=\; (1,\, u_{1,t},\, \ldots,\, u_{m,t})^\top
  \;\in\; \mathbb{R}^{m+1}.
  \end{split}
  \label{eq:Psi_u}
\end{equation}
\end{definition}
\begin{assumption}
\label{ass:ia_membership}
For each $i \in \mathbb{N}$, $\psi_{i,x} \circ F\in \mathcal{H}_x \otimes \mathcal{V}_u$,
i.e., there exist $a_i, b_{ij} \in \mathcal{H}_x$ such that
\begin{equation}
\begin{split}
  \psi_{i,x}(F(x,u))
  \;&=\; a_i(x)\,\psi_{0,u}(u)
        \;+\; \sum_{j=1}^m b_{ij}(x)\,\psi_{j,u}(u).
  \end{split}
  \label{eq:ia_membership}
\end{equation}
\end{assumption}

\begin{theorem}
\label{thm:relaxed}
Let $\{\psi_{i,x}\}_{i=1}^\infty$ be a Riesz basis for
$\mathcal{H}_x$, let $\{\psi_{j,u}\}_{j=0}^m$ be as in
Definition~\ref{def:psi_u}, and let
Assumption~\ref{ass:ia_membership} hold. Let
$\{e_j\}_{j=0}^m$ denote the standard basis of
$\mathbb{R}^{m+1}$. Then there exist bounded operators
$\cK:\ell^2\to\ell^2$, $\cA:\ell^2\to\ell^2$ and $\cN_j:\ell^2\to\ell^2$,
$j=1,\ldots,m$, such that
\begin{equation}
  z_{t+1} \;=\; \cK(z_t\otimes v_t)
  \;=\; \cA z_t + \sum_{j=1}^m \cN_j z_t\, u_{j,t},
  \label{eq:bilinear}
\end{equation}
with $\cK = \cA\otimes e_0^\top + \sum_{j=1}^m
\cN_j\otimes e_j^\top$.
\end{theorem}
The proof is given in Appendix~\ref{a:0}.
\begin{remark}[Relation with exact bilinearization]
\label{rem:goswami}
Assumption~\ref{ass:ia_membership} can be regarded as the discrete-time
counterpart of the necessary and sufficient conditions
of~\cite{Bruder2020} for exact bilinearization of
continuous-time control-affine systems, requiring the
autonomous part and input coefficient functions of each
observable to lie in $\mathcal{H}_x$.
In~\cite{GoswamPaley2022}, these conditions are enforced
via closure under Lie derivatives along the input vector
fields. The linear-in-control model of~\cite{Korda_2018_Koop},
$z_{t+1} = Az_t + Bu_t$, corresponds to the special case
$b_{ij}(x) = [B]_{ij}$ constant, which is generally not
satisfied when $g_j(x)$ is state-dependent.
\end{remark}
\subsection{Data-driven finite-dimensional approximation}
\label{sec:3:1}
In this section we provide brief guidelines for applying the EDMD approach \cite{Williams_EDMD_2015} to the developed generalized Koopman operator. The main idea of EDMD in its simplest form, i.e., least-squares regression, is to
generate a set of data snapshots $\{(x_t,u_t,x_{t+1})\}_{t\in\Nset_{[0,T-1]}}$, $T\in\Nset_{\geq 1}$ using the nonlinear dynamical system \eqref{eq:3:1}, i.e., such that it holds that  $x_{t+1}=F(x_t,u_t)$ for all $t$. Then, letting $z_t=\Psi_x(x_t)\in\Rset^{n_z}$, $v_t=\Psi_u(u_t)\in\Rset^{n_v}$, for any $t\in\Nset$, where the vectors of observables are now finite-dimensional, we can define the data matrices:
\begin{subequations}
\label{eq:d:1}
\begin{align}
Z_X&=\begin{bmatrix}\Psi_x(x_0)&\Psi_x(x_1)&\ldots&\Psi_x(x_{T-1})\end{bmatrix}\nonumber\\
&=\begin{bmatrix}z_0&z_1&\ldots&z_{T-1}\end{bmatrix}\in\Rset^{n_z\times T},\label{eq:d:1a}\\
V_U&=\begin{bmatrix}\Psi_u(u_0)&\Psi_u(u_1)&\ldots&\Psi_u(u_{T-1})\end{bmatrix}\nonumber\\
&=\begin{bmatrix}v_0&v_1&\ldots&v_{T-1}\end{bmatrix}\in\Rset^{n_v\times T},\label{eq:d:1b}\\
Z_X^+&=\begin{bmatrix}\Psi_x(x_1)&\Psi_x(x_2)&\ldots&\Psi_x(x_{T})\end{bmatrix}\nonumber\\
&=\begin{bmatrix}z_1&z_2\ldots&z_T\end{bmatrix}\in\Rset^{n_z\times T}.\label{eq:d:1c}
\end{align}
\end{subequations}
Above $n_z,n_v\in\Nset$ correspond to the number of state and input observable functions, respectively.
\begin{figure}[!t]
\centerline{\includegraphics[width=1\columnwidth]{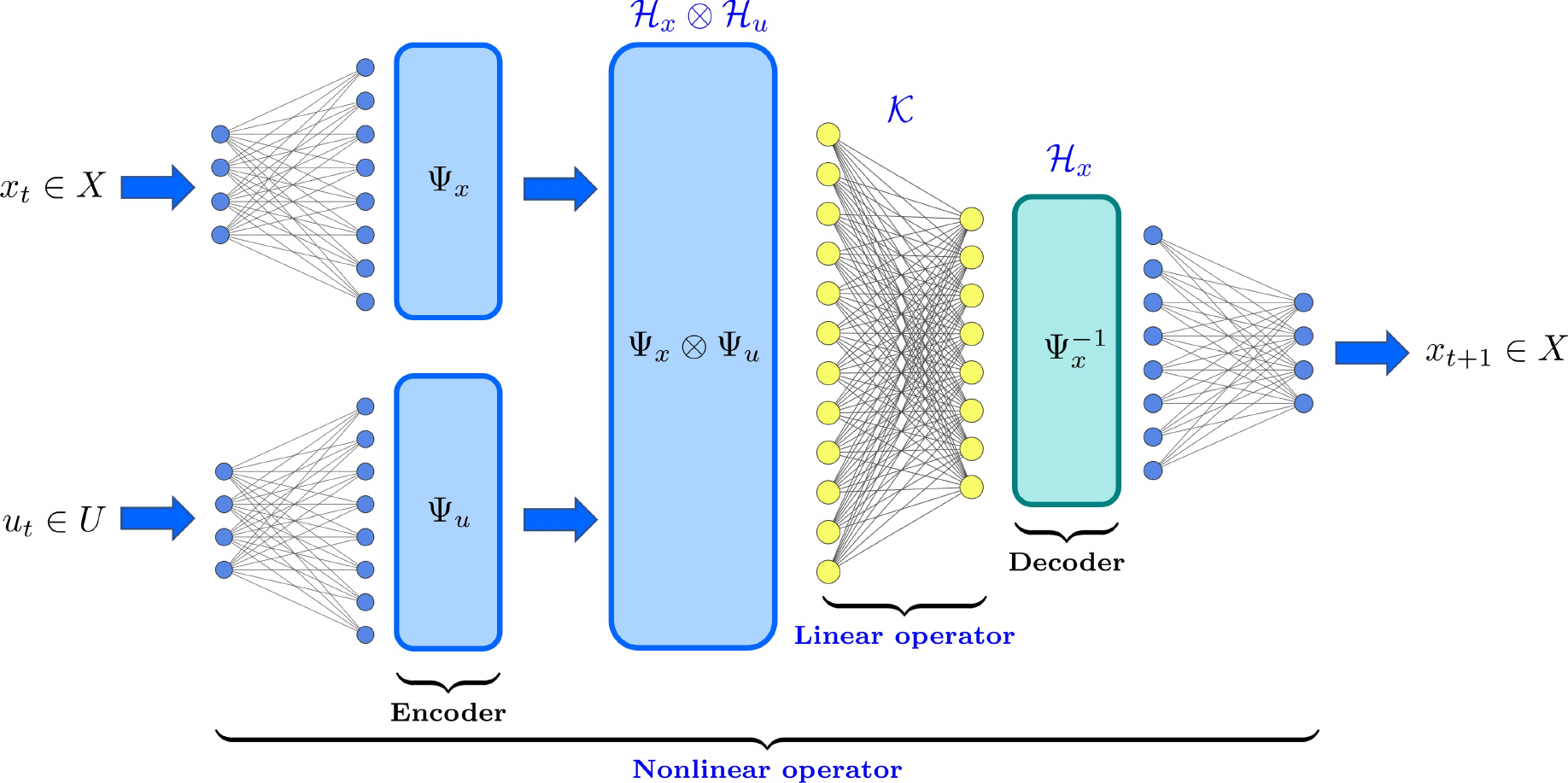}}
\caption{Generalized Koopman operator learning architecture: \emph{(i)}~State and input at time $t$ are lifted via Encoders; \emph{(ii)}~Tensor product of the resulting lifted state and input vectors; \emph{(iii)}~Koopman linear operator advances the lifted state to $t+1$; \emph{(iv)}~State at time $t+1$ is retrieved using a Decoder.}
\label{fig3}
\end{figure}
Next, differently from standard EDMD for systems with control input, which computes matrices $(A,B)$ such that $Z_X^+\approx AZ_X+BV_U$, our aim is to compute a matrix $\hat \cK\in\Rset^{n_z\times n_zn_v}$ such that
\begin{equation}
\label{eq:d:2}
\begin{split}
Z_X^+&=\hat\cK (Z_X\odot V_U)\\
&=\hat \cK \begin{bmatrix}z_0\otimes v_0&z_1\otimes v_1&\ldots&z_{T-1}\otimes v_{T-1}\end{bmatrix}.
\end{split}
\end{equation}
Above $\odot$ denotes the Khatri-Rao product and $\otimes$ denotes the Kronecker product of two columns/vectors. Assuming that the matrix $(Z_X\odot V_U)$ has full-row rank, we can compute a finite-dimensional approximation of $\cK$ as
\begin{equation}
\label{eq:d:3}
\hat\cK:=Z_X^+(Z_X\odot V_U)^\dagger.
\end{equation}
Using a randomised SVD~\cite{Halko2011} applied to the
Khatri-Rao matrix $(Z_X\odot V_U)\in\mathbb{R}^{n_zn_v\times T}$,
the computational cost reduces to $O(n_z n_v \cdot T \cdot r)$, where $r\ll n_zn_v$
is the target rank, making the method scalable to large datasets via GPU implementation.

Figure~\ref{fig3} illustrates the developed architecture for learning generalized Koopman operators, in which kernel functions, deep neural networks or Takens delay embeddings can be used as observable functions. 
\begin{remark}[Convergence analysis]
\label{rem:edmd_error}
Truncating to $n_z$ state and $n_v$ input observables,
the total EDMD approximation error decomposes as
\begin{align}
  \|\cK &- \hat{\cK}^{(n_z,n_v)}\|
  \nonumber\\&\leq
  \underbrace{\|\cK - \cK^{(n_z,n_v)}\|}_{\text{Galerkin error}}
  +
  \underbrace{\|\cK^{(n_z,n_v)} -
  \hat{\cK}^{(n_z,n_v)}\|}_{\text{data error}},
  \label{eq:error_decomp}
\end{align}
where $\cK^{(n_z,n_v)}$ is the exact Galerkin
projection onto the truncated subspace.
The Galerkin error vanishes as $n_z,n_v\to\infty$
by completeness of the Riesz bases and boundedness
of $\cK$ (Remark~\ref{rem:K_bounded}), provided
the truncated subspaces are nested with dense union
in $\ell^2$; the analogous result for orthonormal
bases is proved in~\cite{Korda_2018_convergence}.
When the observables are universal
kernels~\cite{Micchelli_2006} on compact $X$, $U$
and $F$ is sufficiently smooth, the functions
$\psi_{i,x}\circ F$ lie in the product RKHS
$\mathcal{H}_{k_x}\otimes\mathcal{H}_{k_u}$, which
is dense in $C(X\times U)$ by universality;
if additionally $\mu_x\otimes\mu_u$ is a regular
Borel measure, then $C(X\times U)$ is dense in
$L^2(X\times U,\mu_x\otimes\mu_u)$~\cite[Theorem~3.14]{Rudin1987},
and the Galerkin error decays at the kernel
interpolation rate~\cite{Wendland2004}.
The data error is $O(T^{-1/2})$ by concentration
inequalities for empirical inner products when the
data triples are independently
sampled~\cite{KlusNuske2020}.
\end{remark}
\begin{remark}
Finite-dimensional EDMD approximation with
$n_z$ state and $n_v$ input observables does not require that the
generalized invariance condition
$\psi_{i,x}\circ F\in\mathcal{H}$ holds for the truncated system, i.e., the finite matrix
$\hat{\cK}^{(n_z,n_v)}$ is always well-defined as
the Galerkin projection of the infinite-dimensional
$\cK$ onto the finite subspace
$\mathcal{V}_{n_z}\otimes\mathcal{V}_{n_v}$.
Boundedness of $\cK$ (Remark~\ref{rem:K_bounded})
and completeness of the Riesz basis
 guarantee that the Galerkin truncation error
converges to $0$ as $n_z,n_v\to\infty$~\cite{Korda_2018_convergence}.
Invariance is not required for this convergence, which is consistent with classical EDMD. Alternatively, one can aim to learn a finite-dimensional invariant subspace, as done in \cite{Brunton_PLOS}. In finite-dimensional EDMD, one can employ the extended bilinear Koopman model
\begin{equation*}
\begin{split}
&z_{t+1}=A z_t+B v_t+\hat\cK(z_t\otimes v_t)=\\
&A\begin{bmatrix}x_t\\\Psi_x(x_t)\end{bmatrix}+B\begin{bmatrix}u_t\\\Psi_u(u_t)\end{bmatrix}+\hat\cK\left(\begin{bmatrix}x_t\\\Psi_x(x_t)\end{bmatrix}\otimes\begin{bmatrix}u_t\\\Psi_u(u_t)\end{bmatrix}\right).
\end{split}
\end{equation*}
Including unlifted states and inputs yields injectivity and it enables cost functions and constraints in the original system variables; the linear terms can also compensate for residuals.
\end{remark}

\section{The nonlinear fundamental lemma: \\ a bilinear-Koopman approach}
\label{sec:4}
In order to derive an analogue of the fundamental lemma \cite{WillemsRapisarda2005} for nonlinear systems, consider next the nonlinear system \eqref{eq:3:1} with a state-dependent output added, i.e.,
\begin{equation}
\label{eq:4:1}
\begin{split}
x_{t+1}&=F(x_t,u_t),\quad t\in\Nset\\
y_t&=h(x_t),
\end{split}
\end{equation}
where $h:X\rightarrow Y$, $Y\subseteq\Rset^p$ is a suitable mapping. We assume that $h$ is Lebesgue integrable and we introduce the Hilbert space $\cH_y=L^2(Y,\mu_y)$ on $Y\subseteq\Rset^p$ with the standard inner product  $\langle\cdot,\cdot\rangle_{\cH_y}:\cH_y\times\cH_y\rightarrow\Kset$.
\subsection{Koopman representations with lifted outputs}
Next, we consider a composition operator $C_h : \cH_y \to \cH_x$ defined by $(C_h\varphi)(x) = \varphi(h(x))$ for a measurable output map $h: X\to Y$ and we directly prove existence of an infinite-dimensional matrix $\cC$ that is the coordinate representation of $C_h$ in a Riesz basis, as done for the generalized Koopman operator. To this end we introduce the following set of $\Kset$-valued observable functions for the output of system \eqref{eq:4:1}, i.e., $\{\psi_{i,y}\}_{i=1}^\infty$, with $\psi_{i,y}\in\cH_y$ for all $i\in\Nset$. Furthermore, we construct a lifted output vector $\Psi_y(y_t)$ by stacking the corresponding observable functions, e.g., $\Psi_y(y_t)=\col(\psi_{1,y}(y_t),\psi_{2,y}(y_t),\ldots )$.
\begin{proposition}
\label{prop:1y}
Let $\{\psi_{i,x}\}_{i=1}^\infty$ be an orthonormal basis in $\cH_x$ and let $\{\phi_{i}\}_{i=1}^\infty$ and $\{\psi_{i,y}\}_{i=1}^\infty$ be two orthonormal bases in $\cH_y$. Let the output map $h:X\rightarrow Y$ in \eqref{eq:4:1} satisfy:
\begin{equation}
\label{eq:4:3}
\phi_{i}\circ h\in\cH_x,\quad i\in\Nset.
\end{equation}
Then there exists an infinite-dimensional matrix $\cC$ such that
\begin{equation}
\label{eq:4:4}
\Psi_y[h(x)]=\cC\Psi_x(x),
\end{equation}
where
\begin{equation*}
\begin{split}
\cC=&\sum_{j=1}^\infty \begin{bmatrix}\langle\psi_{1,y},\phi_{j}\rangle_{\cH_y}\\\langle\psi_{2,y},\phi_{j}\rangle_{\cH_y}\\\vdots\end{bmatrix}\cdot\\&\cdot\begin{bmatrix}\langle\phi_{j}\circ h,\psi_{1,x}\rangle_{\cH_x} & \langle\phi_{j}\circ h,\psi_{2,x}\rangle_{\cH_x}& \ldots\end{bmatrix}.
\end{split}
\end{equation*}
\end{proposition}

As done for $\cK$, if we choose $\psi_{i,y}=\phi_{i}$, $i\in\Nset$, the output matrix $\cC$ can be simplified to
\begin{equation}
\label{eq:4:7}
\bar \cC=\begin{bmatrix}\langle\phi_1\circ h,\psi_{1,x}\rangle_{\cH_x}&\langle\phi_1\circ h,\psi_{2,x}\rangle_{\cH_x}&\ldots\\
\langle\phi_2\circ h,\psi_{1,x}\rangle_{\cH_x}&\langle\phi_2\circ h,\psi_{2,x}\rangle_{\cH_x}&\ldots\\
\vdots&\vdots&\ddots\end{bmatrix}
\end{equation}
and the condition on the set of observable functions $\{\psi_{i,y}\}_{i=1}^\infty$ can be relaxed to a linearly independent and complete set of functions that form a Riesz basis in $\cH_y$. First, notice that we can expand  $\psi_{i,y}=\sum_{j=1}^\infty\langle\psi_{i,y},\phi_j\rangle_{\cH_y}\phi_j$, which can be compactly written as
\begin{equation}
\label{eq:3:cy}
\begin{bmatrix}\psi_{1,y}\\\psi_{2,y}\\\vdots\end{bmatrix}=
\underbrace{\begin{bmatrix}\langle\psi_{1,y},\phi_1\rangle_{\cH_y}&\langle\psi_{1,y},\phi_2\rangle_{\cH_y}&\ldots\\\langle\psi_{2,y},\phi_1\rangle_{\cH_y}&\langle\psi_{2,y},\phi_2\rangle_{\cH_y}&\ldots\\\vdots&\vdots&\ddots\end{bmatrix}}_{:=T_y}\begin{bmatrix}\phi_1\\\phi_2\\\vdots\end{bmatrix},
\end{equation}
where the infinite-dimensional matrix $T_y$ is nonsingular.
\begin{corollary}
\label{cor:1y}
If the observables $\{\psi_{i,x}\}_{i=1}^\infty$ and $\{\psi_{i,y}\}_{i=1}^\infty$ form a Riesz basis in $\cH_x$ and $\cH_y$, respectively, then the infinite-dimensional matrix $\cC$ in \eqref{eq:4:4} is given by
\begin{equation}
\label{eq:4:8}
\cC=T_y\bar \cC T_x^{-1},
\end{equation}
with $T_x, T_y, \bar \cC$ defined in \eqref{eq:3:cx}, \eqref{eq:3:cy} and \eqref{eq:4:7}, respectively.
\end{corollary}
The version of Theorem~\ref{thm:1} for the output is given next.
\begin{theorem}
\label{thm:1y}
Let $\{\psi_{i,x}\}_{i=1}^\infty$ and $\{\psi_{i,y}\}_{i=1}^\infty$ be independent and complete sets of Riesz basis functions spanning $\cH_x$ and $\cH_y$, respectively. Let the output map $h:X\rightarrow Y$ in \eqref{eq:4:1} satisfy:
\begin{equation}
\label{eq:3:8y}
\psi_{i,y}\circ h\in\cH_x,\quad i\in\Nset.
\end{equation}
Then there exists an infinite-dimensional matrix $\cC=Q_yR_x^{-1}$ such that
\begin{equation}
\label{eq:3:9y}
\Psi_y[h(x)]=\cC\Psi_x(x),
\end{equation}
with the matrices $R_x$ and  $Q_y$ given by:
\begin{equation}
\label{eq:3:10y}
R_x=\begin{bmatrix}\langle\psi_{1,x},\psi_{1,x}\rangle_{\cH_x} &\langle\psi_{1,x},\psi_{2,x}\rangle_{\cH_x}&\ldots\\\langle\psi_{2,x},\psi_{1,x}\rangle_{\cH_x}&\langle\psi_{2,x},\psi_{2,x}\rangle_{\cH_x}&\ldots\\\vdots&\vdots&\ddots\end{bmatrix}
\end{equation}
and
\begin{equation}
\label{eq:3:11y}
Q_y=\begin{bmatrix}\langle\psi_{1,y}\circ h,\psi_{1,x}\rangle_{\cH_x} &\langle\psi_{1,y}\circ h,\psi_{2,x}\rangle_{\cH_x}&\ldots\\\langle\psi_{2,y}\circ h,\psi_{1,x}\rangle_{\cH_x}&\langle\psi_{2,y}\circ h,\psi_{2,x}\rangle_{\cH_x}&\ldots\\\vdots&\vdots&\ddots\end{bmatrix}.
\end{equation}
\end{theorem}
For the proofs of Proposition~\ref{prop:1y} and Theorem~\ref{thm:1y} we refer to Appendix~\ref{a:1} and Appendix~\ref{a:2}, respectively.
\begin{remark}[Boundedness of $C_h$ and
local observability]
\label{rem:output_map}
Analogous conditions to
Assumptions~\ref{ass:nonsing}--\ref{ass:bounded}
can be derived to guarantee that the composition operator
$C_h : \mathcal{H}_y \to \mathcal{H}_x$
defined by $(C_h\varphi)(x) = \varphi(h(x))$
for a measurable output map $h: X\to Y$ is a well-defined bounded linear operator.
Specifically, non-singularity and boundedness of
the Radon-Nikodym derivative $w_h = dh_*\mu_x/d\mu_y
\in L^\infty(Y,\mu_y)$ are sufficient for
$C_h$ to be bounded with
$\|C_h\|_{\mathcal{H}_y\to\mathcal{H}_x}
\leq \|w_h\|_{L^\infty}^{1/2}$, by the same
argument as in Theorem~\ref{thm:bounded_CF}. For $C^1$ maps $h: X\to Y$ with
$\dim(Y)\leq\dim(X)$, these conditions reduce via the coarea formula \cite[Section~3.4]{EvansGariepy1992} to
\begin{equation*}
  J_h(x) \;:=\;
  \sqrt{\det\!\left(
  \frac{\partial h}{\partial x}(x)
  \left(\frac{\partial h}{\partial x}(x)
  \right)^\top\right)}\geq c>0,
\end{equation*}
a.e. on $X$. This requires that $\partial h/\partial x(x)$ has full row
rank $\dim(Y)$ a.e., i.e., distinct states are locally distinguishable through $h$ or equivalently, \emph{uniform local observability} holds. Furthermore, similarly as derived for $\cK$, it can be established that the infinite-dimensional matrix $\cC$ defines a well-defined bounded linear operator on $\ell^2$.
\end{remark}

Let $w_t:=\Psi_y(y_t)$ and define the following exact infinite-dimensional bilinear Koopman system with a linear output:
\begin{equation}
\label{eq:4:2}
\begin{split}
z_{t+1}&=\cK (z_t\otimes v_t),\\
w_t&=\Psi_y(y_t)=\Psi_y[h(x_t)]=\cC\Psi_x(x_t)\\
&=\cC z_t.
\end{split}
\end{equation}
Note that Koopman models with lifted outputs have been previously considered in \cite{Park_2024} to study observability properties.

Now we are ready to define a relation between the solutions of the nonlinear system \eqref{eq:4:1} and the infinite-dimensional bilinear Koopman system \eqref{eq:4:2}.
\begin{theorem}
\label{thm:2}
Let $\{\psi_{i,x}\}_{i=1}^\infty$, $\{\psi_{i,u}\}_{i=1}^\infty$ and $\{\psi_{i,y}\}_{i=1}^\infty$ be independent and complete sets of Riesz basis functions spanning $\cH_x$, $\cH_u$ and $\cH_y$, respectively. Let the state transition map $F : X\times U\rightarrow X$ and the output map $h:X\rightarrow Y$ in \eqref{eq:3:1} satisfy:
\begin{equation}
\label{eq:assum}
\begin{split}
&\psi_{i,x}\circ F\in\cH,\quad i\in\Nset,\\
&\psi_{i,y}\circ h\in\cH_x,\quad i\in\Nset.
\end{split}
\end{equation}
Let $\{\varphi_i\}_{i=1}^\infty$ and $\{\phi_i\}_{i=1}^\infty$ be orthonormal bases in $\cH_x$, $\cH_y$, and let $\cK$ and $\cC$ be defined as in \eqref{eq:3:7}, \eqref{eq:4:8}. Assume that the state and output lifting mappings $\Psi_x(x)=\col(\psi_{1,x}(x),\psi_{2,x}(x),\ldots)$ and $\Psi_y(y)=\col(\psi_{1,y}(y),\psi_{2,y}(y),\ldots)$ are injective.

Then $\{x_t,u_t,y_t\}_{t\in\Nset}$  is a solution of system \eqref{eq:4:1} if and only if $\{z_t,v_t,w_t\}_{t\in\Nset}$ is a solution of system \eqref{eq:4:2}.
\end{theorem}
\begin{proof}
Suppose that $\{x_t,u_t,y_t\}_{t\in\Nset}$ is a solution of system \eqref{eq:4:1}. Then from Proposition~\ref{prop:1}-Corollary~\ref{cor:1} and Proposition~\ref{prop:1y}-Corollary~\ref{cor:1y} we obtain that $\{z_t,v_t,w_t\}_{t\in\Nset}=\{\Psi_x(x_t),\Psi_u(u_t),\Psi_y(y_t)\}_{t\in\Nset}$ is a solution of system \eqref{eq:4:2}. Note that this part of the proof does not require that the mappings $\Psi_x, \Psi_y$ are injective.

Conversely, let $\{x_t,u_t,y_t\}_{t\in\Nset}$ be such that $\{z_t,v_t,w_t\}_{t\in\Nset}$ with $z_t=\Psi_x(x_t)$, $v_t=\Psi_u(u_t)$ and $w_t=\Psi_y(y_t)$ is a solution of system \eqref{eq:4:2}. For all $t\in\Nset$ let $\hat x_{t+1}=F(x_t,u_t)$ and let $\hat y_t=h(x_t)$ and suppose that $\hat x_{t+1}\neq x_{t+1}$ and $\hat y_t\neq y_t$ for some $t\in\Nset$. Then, by Corollary~\ref{cor:1} and Corollary~\ref{cor:1y} we have that $\Psi_x(\hat x_{t+1})=\cK\Psi(x_t,u_t)=\Psi_x(x_{t+1})$ and $\Psi_y(\hat y_t)=\cC\Psi_x(x_t)=\Psi_y(y_t)$. If the mappings $\Psi_x$ and $\Psi_y$ are injective, then it must hold that $\hat x_{t+1}=x_{t+1}$ and $\hat y_t = y_t$ and thus, $\{x_t,u_t,y_t\}_{t\in\Nset}$ is a solution of system \eqref{eq:4:1}.
\end{proof}

\subsection{The nonlinear fundamental lemma}
Let $N_d \in \mathbb{N}$ data sets be
collected from system \eqref{eq:4:1}, each from initial condition $x_0^{(i)}$,
$i=1,\ldots,N_d$, with trajectory length $T_i+N$,
giving lifted trajectories
$z_t^{(i)} = \Psi_x(x_t^{(i)}) \in \ell^2$,
$v_t^{(i)} = \Psi_u(u_t^{(i)}) \in \ell^2$,
$w_t^{(i)} = \Psi_y(y_t^{(i)}) \in \ell^2$.
The $N$-step lifted sequence from data set $i$ at time $s$ is
\begin{equation*}
  \Phi_s^{(i)} \;:=\;
  \operatorname{col}\!\left(
    z_s^{(i)}\otimes v_s^{(i)},\;\ldots,\;
    z_{s+N-1}^{(i)}\otimes v_{s+N-1}^{(i)}
  \right) \;\in\; \ell^2.
\end{equation*}
\begin{definition}[Multiple-data-sets Hankel operator]
\label{def:hankel_multi}
Let $\mathcal{T} := \bigoplus_{i=1}^{N_d}\mathcal{G}_i$
with $\mathcal{G}_i = \mathbb{R}^{T_i}$ for
$T_i<\infty$ and $\mathcal{G}_i = \ell^2$ for
$T_i=\infty$, with elements
$g = (g^{(1)},\ldots,g^{(N_d)})$,
$g^{(i)} = (g_0^{(i)},g_1^{(i)},\ldots)$.
Define the \emph{multiple-data-sets Hankel operator} as
\begin{equation}
  \mathcal{F}_N : \mathcal{T}\to\ell^2,
  \qquad
  \mathcal{F}_N g
  \;:=\; \sum_{i=1}^{N_d}\sum_{s=0}^{T_i-1}
         g_s^{(i)}\,\Phi_s^{(i)},
  \label{eq:hankel_multi}
\end{equation}
where, for $T_i=\infty$, the inner sum converges
in $\ell^2$ by square-summability of
$g^{(i)}\in\ell^2$.
\end{definition}
\begin{definition}
\label{def:behavior}
The \emph{$N$-step truncated behavior} of the
bilinear Koopman system \eqref{eq:4:2} is defined
as the set of all $N$-step admissible data sequences
$\Phi_t := \operatorname{col}\!\left(
z_t\otimes v_t,\,z_{t+1}\otimes v_{t+1},\,\ldots,\,
z_{t+N-1}\otimes v_{t+N-1}\right)\in\ell^2$:
\begin{align}
    &\mathcal{B}_N:=\nonumber\\
  &\{ \Phi_t \in\ell^2   \;|\;  z_t\in\ell^2,\;
  v_{[t,t+N-1]}\in\ell^2,\;
  z_{i+1} = \cK(z_{i}\otimes v_{i})\},\notag\\
  &i=t,\ldots,t+N-1.\label{eq:behavior}
\end{align}
\end{definition}
Let $\mathcal{I} := \{(i,s) :
i=1,\ldots,N_d,\; s=0,\ldots,T_i-1\}$.
\begin{definition}[Persistency of Excitation]
\label{def:PE}
The data collection $\{\Phi_s^{(i)}\}_{(i,s)\in \mathcal{I}}$
is \emph{persistently exciting} of order $N$ if
\begin{equation}
\label{eq:PE_def}
  \mathcal{B}_N \subseteq \mathcal{R}(\mathcal{F}_N),
\end{equation}
where the range of $\mathcal{F}_N:\mathcal{T}\to\ell^2$
is
\begin{equation*}
  \mathcal{R}(\mathcal{F}_N)
  := \left\{
  \sum_{(i,s)\in \mathcal{I}} g_s^{(i)}\Phi_s^{(i)}
  \;:\; g\in\mathcal{T}
  \right\},
\end{equation*}
i.e., the set of all $\ell^2$-convergent linear
combinations of data sequences with
$g\in\mathcal{T}$.
\end{definition}

The key question for deriving the nonlinear
fundamental lemma is whether any admissible
$N$-step trajectory $\Phi\in\mathcal{B}_N$ can
be represented as $\mathcal{F}_Ng$ for some
$g\in\mathcal{T}$. The PE
condition~\eqref{eq:PE_def} alone does not
guarantee well-posedness in infinite dimensions,
since $\mathcal{R}(\mathcal{F}_N)$ may not be
closed, i.e., a sequence
$\Phi^{(n)} = \mathcal{F}_N g^{(n)}$ may converge
in $\ell^2$ to a limit outside
$\mathcal{R}(\mathcal{F}_N)$.
Hence, we strengthen the PE condition as follows.
\begin{assumption}[Quantitative PE]\label{ass:spanning}
The data sequences $\{\Phi_s^{(i)}\}_{(i,s)\in\mathcal{I}}$
form a \emph{frame}~\cite[Definition~5.1.1]{Christensen2016}
for $\mathcal{B}_N$, i.e.\ there exist
$0 < \alpha_F \leq \beta_F < \infty$ such that
\begin{equation}
  \alpha_F\|\Phi\|_{\ell^2}^2
  \;\leq\;
  \sum_{(i,s)\in\mathcal{I}}
  |\langle\Phi_s^{(i)},\Phi\rangle_{\ell^2}|^2
  \;\leq\;
  \beta_F\|\Phi\|_{\ell^2}^2,
  \quad\forall\,\Phi\in\mathcal{B}_N.
  \label{eq:frame_cond}
\end{equation}
\end{assumption}
Under Assumption~\ref{ass:spanning}, the Hankel
operator $\mathcal{F}_N$ in
Definition~\ref{def:hankel_multi} serves as the
synthesis operator of the frame
$\{\Phi_s^{(i)}\}_{(i,s)\in\mathcal{I}}$.
The upper bound in~\eqref{eq:frame_cond} implies
that $\{\Phi_s^{(i)}\}_{(i,s)\in\mathcal{I}}$ is
a Bessel sequence with bound $\beta_F$, so
$\mathcal{F}_N$ is bounded with
$\|\mathcal{F}_N\|^2 \leq
\beta_F$~\cite[Theorem~3.2.3]{Christensen2016}.
The lower bound guarantees that
$\mathcal{R}(\mathcal{F}_N)$ is closed and satisfies
$\mathcal{R}(\mathcal{F}_N) =
\overline{\operatorname{span}}_{\ell^2}
\{\Phi_s^{(i)}\}_{(i,s)\in\mathcal{I}}$~\cite[Lemma~5.2.1]{Christensen2016}.
The adjoint $\mathcal{F}_N^*:\ell^2\to\mathcal{T}$
is the analysis operator
\begin{equation*}
  [\mathcal{F}_N^*\Phi]_s^{(i)}
  \;=\;
  \langle\Phi_s^{(i)},\Phi\rangle_{\ell^2},
  \quad (i,s)\in\mathcal{I}.
\end{equation*}
This yields the frame operator
$S_F := \mathcal{F}_N\mathcal{F}_N^*:
\ell^2\to\ell^2$~\cite[eq.~(5.5)]{Christensen2016},
which is bounded, self-adjoint, and positive
by~\cite[Lemma~5.1.5]{Christensen2016}.
The lower frame bound gives
$\langle S_F\Phi,\Phi\rangle_{\ell^2} \geq
\alpha_F\|\Phi\|_{\ell^2}^2$ for all
$\Phi\in\mathcal{B}_N$, so $S_F$ is boundedly
invertible on $\mathcal{B}_N$ with
$\|S_F^{-1}\| \leq \alpha_F^{-1}$.
\begin{proposition}
\label{prop:frame_spanning}
Assumption~\ref{ass:spanning} implies the PE
condition~\eqref{eq:PE_def}. Moreover, for every
$\Phi\in\mathcal{B}_N$, the equation
$\mathcal{F}_Ng = \Phi$ has the explicit
minimum-norm solution
\begin{equation*}
  g \;=\; \mathcal{F}_N^*S_F^{-1}\Phi
  \;\in\;\mathcal{T}.
\end{equation*}
\end{proposition}
\begin{proof}
For any $\Phi\in\mathcal{B}_N$, $S_F$ is boundedly
invertible on $\mathcal{B}_N$ with
$\|S_F^{-1}\|\leq\alpha_F^{-1}$, so
$S_F^{-1}\Phi\in\ell^2$ is well defined and
$g := \mathcal{F}_N^*S_F^{-1}\Phi\in\mathcal{T}$.
Then
\begin{equation*}
  \mathcal{F}_Ng
  = \mathcal{F}_N\mathcal{F}_N^*S_F^{-1}\Phi
  = S_FS_F^{-1}\Phi
  = \Phi,
\end{equation*}
so $\Phi\in\mathcal{R}(\mathcal{F}_N)$ and the
PE condition~\eqref{eq:PE_def} holds.
\end{proof}
\begin{remark}[Relation to Willems' fundamental lemma]
\label{rem:willems}
For a linear system $x_{t+1}=Ax_t+Bu_t$,
$y_t=Cx_t+Du_t$, the $N$-step
data sequences are defined as
$\Phi_s := \operatorname{col}(u_{[s,s+N-1]},\,
y_{[s+1,s+N]})\in\mathbb{R}^{(m+p)N}$,
and the truncated behavior reduces to the
finite-dimensional subspace
\begin{equation*}
  \begin{split}
  \mathcal{B}_N^{\mathrm{IO}} \;:=\;\{& \Phi_s\in\mathbb{R}^{(m+p)N} \;|\;  \exists\,x_0\in\mathbb{R}^n\;\text{s.t.}\;\\
  &x_{t+1}=Ax_t+Bu_t,\;y_t=Cx_t+Du_t \}
  \end{split}
  \end{equation*}
of dimension $mN+n$.
With a single dataset ($N_d=1$), the Hankel operator
reduces to the finite matrix
$H_{N,T}\in\mathbb{R}^{(m+p)N\times T}$
with columns $\Phi_s$, $s=0,\ldots,T-1$,
and the PE condition
$\mathcal{B}_N^{\mathrm{IO}}\subseteq
\mathcal{R}(H_{N,T})$
reduces to the rank condition
\begin{equation}
  \operatorname{rank}(H_{N,T})
  \;=\; \dim(\mathcal{B}_N^{\mathrm{IO}})
  \;=\; mN+n.
  \label{eq:PE_linear}
\end{equation}
This rank condition is achieved when the input
is PE of order $N+n$ and $(A,B)$ is
controllable~\cite{WillemsRapisarda2005}.
In this finite-dimensional setting, the columns
of $H_{N,T}$ form a frame for
$\mathcal{B}_N^{\mathrm{IO}}$ if and only if
the rank condition~\eqref{eq:PE_linear} holds,
with frame bounds
\begin{equation}
  \alpha_F = \sigma_{\min}^2(H_{N,T}),
  \qquad
  \beta_F = \sigma_{\max}^2(H_{N,T}),
  \label{eq:alpha_F_linear}
\end{equation}
corresponding to the eigenvalues of the
finite-dimensional frame operator
$S_F = H_{N,T}H_{N,T}^\top$.
The quantitative PE bound $\alpha_F$ coincides
with the PE notion introduced
in~\cite{CoulsonVanWaarde2023}.
Thus Definition~\ref{def:PE} generalises
\eqref{eq:PE_linear} to infinite dimensions, and
Assumption~\ref{ass:spanning} is the corresponding
generalisation of the quantitative PE
of~\cite{CoulsonVanWaarde2023}, both fitted to the
bilinear Koopman system~\eqref{eq:4:2}.
\end{remark}
\begin{theorem}[Nonlinear Fundamental Lemma]
\label{thm:fundlemma}
Let $\{\psi_{i,x}\}_{i=1}^\infty$,
$\{\psi_{i,u}\}_{i=1}^\infty$,
$\{\psi_{i,y}\}_{i=1}^\infty$ be Riesz bases for
$\mathcal{H}_x$, $\mathcal{H}_u$, $\mathcal{H}_y$
with $\psi_{i,x}\circ F\in\mathcal{H}$,
$\psi_{i,y}\circ h\in\mathcal{H}_x$ for all
$i\in\mathbb{N}$, let $\Psi_x$, $\Psi_y$ be
injective, and let
Assumption~\ref{ass:spanning} hold for
$\{\Phi_s^{(i)}\}_{(i,s)\in\mathcal{I}}$.
Define the \emph{augmented Hankel operator}
$\mathcal{F}_N^w : \mathcal{T}\to\ell^2$ by
$\mathcal{F}_N^w g :=
\sum_{(i,s)\in\mathcal{I}}
g_s^{(i)}\Phi_s^{w,(i)}$,
where $\Phi_s^{w,(i)} :=
\operatorname{col}(\Phi_s^{(i)},\,
w_{[s+1,s+N]}^{(i)})\in\ell^2$
appends the lifted output sequence to $\Phi_s^{(i)}$.

Then for any $t\in\mathbb{N}$,
$\{x_{[t,t+N]},u_{[t,t+N-1]},y_{[t+1,t+N]}\}$
is a solution of~\eqref{eq:4:1} if and only if
there exists $g\in\mathcal{T}$ such that
\begin{equation}
  \mathcal{F}_N^w g
  \;=\;
  \begin{pmatrix}
    (z\otimes v)_{[t,t+N-1]}\\[4pt]
    w_{[t+1,t+N]}
  \end{pmatrix}.
  \label{eq:fundlemma}
\end{equation}
\end{theorem}
\begin{proof}
Recall that $\mathcal{K}: \ell^2 \to \ell^2$ and
$\mathcal{C}: \ell^2 \to \ell^2$ denote the
coordinate representations of the composition
operators $C_F$ and $C_h$ in corresponding Riesz bases. Define
\begin{equation}
  \mathcal{M}:\ell^2\to\ell^2,\quad \mathcal{M} \;:=\;
  \begin{pmatrix}
    I_{\ell^2}         & 0      & \cdots & 0      \\
    \vdots             &        & \ddots & \vdots \\
    0                  & 0      & \cdots & I_{\ell^2} \\
    \mathcal{C}\mathcal{K} & 0      & \cdots & 0      \\
    \vdots             &        & \ddots & \vdots \\
    0                  & 0      & \cdots & \mathcal{C}\mathcal{K}
  \end{pmatrix},
\end{equation}
with $N$ identity blocks and $N$ blocks of
$\mathcal{C}\mathcal{K}$, where
$\mathcal{C}\mathcal{K}: \ell^2\to\ell^2$
maps the coordinate sequence of
$(z_t\otimes v_t)$ to the coordinate sequence of
$w_{t+1}$ via the Riesz basis representations.
Then $\Phi_s^{w,(i)} = \mathcal{M}\Phi_s^{(i)}$
and $\mathcal{F}_N^w = \mathcal{M}\mathcal{F}_N$.

\emph{($\Rightarrow$)}
If $\{x_{[t,t+N]},u_{[t,t+N-1]},y_{[t+1,t+N]}\}$
is a solution of~\eqref{eq:4:1}, then by
Theorem~\ref{thm:bounded_CF} the lifted trajectory
satisfies $z_{i+1} = \cK(z_i\otimes v_i)$,
$w_i = \cC z_i$, $i=t,\ldots,t+N-1$, so
$(z\otimes v)_{[t,t+N-1]}\in\mathcal{B}_N$.
By Assumption~\ref{ass:spanning} and
Proposition~\ref{prop:frame_spanning} there
exists $g\in\mathcal{T}$ with
$\mathcal{F}_Ng = (z\otimes v)_{[t,t+N-1]}$.
Applying $\cM$ to both sides and using
$w_{i+1} = \cC\cK(z_{i}\otimes v_{i})$ yields
\begin{align*}
  \mathcal{F}_N^w g &= \cM\mathcal{F}_Ng
  = \cM(z\otimes v)_{[t,t+N-1]}\\
  &= \operatorname{col}\!\left(
    (z\otimes v)_{[t,t+N-1]},\,
    w_{[t+1,t+N]}\right).
\end{align*}

\emph{($\Leftarrow$)}
Suppose~\eqref{eq:fundlemma} holds for some
$g\in\mathcal{T}$. Splitting into top and bottom
blocks gives
\begin{align*}
  \mathcal{F}_N g &= (z\otimes v)_{[t,t+N-1]},\\
  \operatorname{diag}(\mathcal{C}\mathcal{K},
  \ldots,\mathcal{C}\mathcal{K})\,\mathcal{F}_N g
  &= w_{[t+1,t+N]}.
\end{align*}
Since $\cK:\ell^2\to\ell^2$ is bounded and linear,
the operator $L:\ell^2\to\ell^2$, $L\Phi_t :=
\operatorname{col}(z_{i+1} -
\cK(z_i\otimes v_i))_{i=t}^{t+N-1}$ is bounded
and linear as the difference of two bounded linear
maps, and $\mathcal{B}_N = \ker(L)$ is therefore
a closed linear subspace of
$\ell^2$, see, e.g.,~\cite[proof of Thm.~4.12]{Rudin1987}.
Since each data sequence $\Phi_s^{(i)}\in\mathcal{B}_N$
by construction, and $\mathcal{B}_N$ is closed,
the top block
$(z\otimes v)_{[t,t+N-1]} = \mathcal{F}_N g
\in\mathcal{B}_N$, so the Koopman dynamics
$z_{i+1} = \cK(z_i\otimes v_i)$ hold for
$i=t,\ldots,t+N-1$. The bottom block then gives
$w_{i+1} = \cC\cK(z_i\otimes v_i) = \cC z_{i+1}$
for each $i=t,\ldots,t+N-1$.
Hence the full sequence
$\{z_{[t,t+N]}, v_{[t,t+N-1]}, w_{[t+1,t+N]}\}$
is a solution of system~\eqref{eq:4:2}, and
by Theorem~\ref{thm:2}, injectivity of $\Psi_x$
and $\Psi_y$ uniquely recovers
$\{x_{[t,t+N]}, u_{[t,t+N-1]}, y_{[t+1,t+N]}\}$
as a solution of~\eqref{eq:4:1}.
\end{proof}
\begin{remark}[Data generation guidelines]
\label{rem:data_generation}
For universal kernels $k_x$, $k_u$ on compact
$X$, $U$~\cite{Micchelli_2006}, the kernel feature
maps generate function systems $\{\psi_{i,x}\}_{i=1}^\infty$,
$\{\psi_{l,u}\}_{l=1}^\infty$ that can be chosen to form Riesz
bases for $\mathcal{H}_{k_x}$ and $\mathcal{H}_{k_u}$,
with tensor products $\{\psi_{i,x}\psi_{l,u}\}$
forming a Riesz basis for $\mathcal{H}=
\mathcal{H}_{k_x}\otimes\mathcal{H}_{k_u}$.
Using such kernels with well-distributed initial
conditions $\{x_0^{(i)}\}_{i=1}^{N_d}$ in $X$ and persistently
exciting input sequences of order $N$ in $U$
creates dense subspaces in $\ell^2$, which favors
the construction of a frame for $\mathcal{B}_N$
and hence the satisfaction of
Assumption~\ref{ass:spanning} with $\alpha_F>0$.
Multiple data sets with independently chosen
initial conditions and inputs further facilitate
the required coverage of $\mathcal{B}_N$.
In the finite-dimensional EDMD case, this yields
a strong quantifiable PE condition for bilinear
systems analogous to merging~\cite{CoulsonVanWaarde2023},
the bilinear fundamental lemma~\cite{Yuan_Cortes_Bilinear}
and its multiple data sets
extension~\cite{Timm_bilinear}.
\end{remark}

\begin{remark}[Multi-step approximation]
\label{rem:multistep}
In the finite-dimensional case with $n_z$ state
observables and $n_v$ input observables,
the lifted trajectories $z_t\in\mathbb{R}^{n_z}$,
$v_t\in\mathbb{R}^{n_v}$ and the Hankel operator
reduces to the finite matrix
$\mathcal{F}_N\in\mathbb{R}^{n_zn_vN\times
\sum_i T_i}$. If $\mathcal{F}_N$ has full row rank,
the system~\eqref{eq:fundlemma} has the
minimum-norm least-squares solution
\begin{equation}
  g^\ast \;:=\;
  \mathcal{F}_N^\dagger
  (z\otimes v)_{[t,t+N-1]},
  \label{eq:ls_sol}
\end{equation}
where $\mathcal{F}_N^\dagger$ is the
Moore-Penrose pseudoinverse, and the predicted
lifted output sequence is
\begin{equation}
\begin{split}
  w^\ast_{[t+1,t+N]}
  =\underbrace{
  \begin{bmatrix}
    w_{[s+1,s+N]}^{(i)}
  \end{bmatrix}_{(i,s)\in\mathcal{I}}
  \mathcal{F}_N^\dagger
  }_{\text{finite-dim.\ approx.\ of }
  \operatorname{diag}(\cC\cK,\ldots,\cC\cK)}
  (z\otimes v)_{[t,t+N-1]}.
  \end{split}
  \label{eq:ms_pred}
\end{equation}
This generalizes EDMD to multi-step
prediction directly from data, without explicitly
identifying $\cK$, and is well-suited for
Koopman MPC with prediction
horizon $N$~\cite{Korda_2018_Koop}.
\end{remark}

\section{Illustrative examples}
\label{sec:5}
All results in this section were generated using a 13th Gen Intel® CoreTM i7-1370P CPU running at 1.90 GHz on a Lenovo Windows Laptop and using Matlab 2023a.
\subsection{Van der Pol Oscillator}
We illustrate the developed generalized Koopman operator method for the Van der Pol oscillator. The training data set is generated by simulating the forced Van der Pol oscillator
\begin{equation}
\label{eq:5:1}
    \dot{x}_1 = x_2, \qquad
    \dot{x}_2 = \mu(1-x_1^2)x_2 - x_1 + u,
\end{equation}
with $\mu = 1.2$, discretized at sampling time $T_s = 0.1$\,s via \texttt{ode45}.

A total of $N_{d} = 100$ initial conditions are drawn from
$X=[-2,2]^2$ using a scrambled Halton sequence, and $N_{\mathrm{seq}} = 10$
persistently exciting multisine input sequences of length $T_{\mathrm{seq}} = 16$
steps, with amplitudes in $U=[-1,1]$, are applied to each initial condition,
yielding a database of $N_{d} \times N_{\mathrm{seq}} \times
T_{\mathrm{seq}} = 16000$ one-step transition pairs
$(x_t, u_t, x_{t+1})$.
From this database, $N_{\mathrm{train}} = 15000$ training pairs and
$N_{\mathrm{val}} = 100$ validation pairs are selected using $D^2$ landmark sampling~\cite{Arthur2007}
on the joint state-input space $X\times U$. Two kernel-based EDMD variants are implemented and compared.
\paragraph{Standard joint-kernel EDMD - kernelized KIC~\cite{Brunto_Koopmanu}}
The observable map lifts the joint state-input pair as
\begin{equation}
    \Psi(x,u) = \bigl[k_\sigma([x;u], c_1), \ldots,
    k_\sigma([x;u], c_{N_b})\bigr]^\top \in \mathbb{R}^{N_b},
\end{equation}
where $c_i = [x_i; u_i] \in \mathbb{R}^{n+m}$ are $N_b = 5000$ or $N_b = 10000$
randomly selected joint landmark pairs and $k_\sigma$ is the inverse quadratic (IQ) kernel, i.e.,
\begin{equation*}
    k_\sigma(z, c_i) = \frac{1}{1 + \|z - c_i\|^2 / \sigma^2},
    \qquad z = \begin{bmatrix} x \\ u \end{bmatrix} \in \mathbb{R}^{n+m},
\end{equation*}
with a single bandwidth $\sigma = 1$. The regression target is
$\Psi(x^+, u)$, i.e., the same joint lifting evaluated
at the next state while keeping the current input, so that the
identified operator $\hat\cK \in \mathbb{R}^{N_b \times N_b}$ satisfies
\begin{equation*}
    \hat\cK= \Psi_{Z^+}\Psi_Z^\top(\Psi_Z \Psi_Z^\top + \gamma I)^{-1},
\end{equation*}
where $\Psi_Z = [\Psi(z_1),\ldots,\Psi(z_N)]$
with $z_k = [x_k; u_k]$ and
$\Psi_{Z^+} = [\Psi(z_1^+),\ldots,\Psi(z_N^+)]$
with $z_k^+ = [x_k^+; u_k]$; $\gamma=1e-6$.

\paragraph{Khatri--Rao EDMD - Generalized Koopman}
The observable map instead factorizes the state and input liftings
independently via
\begin{equation*}
    \Psi(x,u)
    = \Psi_x(x) \otimes \Psi_u(u)
    \in \mathbb{R}^{M_x M_u},
\end{equation*}
using $M_x = 100$ or $M_x=200$ state landmarks with bandwidth $\sigma_x = 1$ and
$M_u = 50$ input landmarks with bandwidth $\sigma_u = 0.54$, selected
independently from the training data. The state and input kernels are defined independently on their
respective spaces, with centers
$c_i^x \in \mathbb{R}^{n}$ and $c_j^u \in \mathbb{R}^{m}$:
\begin{align*}
    k_{\sigma_x}(x, c_i^x) &= \frac{1}{1 + \|x - c_i^x\|^2 / \sigma_x^2},
    \\
    k_{\sigma_u}(u, c_j^u) &= \frac{1}{1 + \|u - c_j^u\|^2 / \sigma_u^2},
\end{align*}
with $\sigma_x = 1$ and $\sigma_u = 0.54$, giving the vectors of observables
\begin{align*}
    \Psi_x(x) &= \bigl[k_{\sigma_x}(x,c_1^x),\ldots,
    k_{\sigma_x}(x,c_{M_x}^x)\bigr]^\top \in \mathbb{R}^{M_x}, \\
    \Psi_u(u) &= \bigl[k_{\sigma_u}(u,c_1^u),\ldots,
    k_{\sigma_u}(u,c_{M_u}^u)\bigr]^\top \in \mathbb{R}^{M_u}.
\end{align*}
This implies that the total number of features is the same for both methods, i.e., $5000$ and $10000$, respectively, selected out of $15000$ data points, which ensures a well-posed regression problem. The regression target is the state-only lifting $\Psi_x(x^+) \in \mathbb{R}^{M_x}$,
yielding $\hat\cK \in \mathbb{R}^{M_x M_u \times M_x}$, computed as in equation \eqref{eq:d:3} with regularization weight $\gamma=1e-6$.

For each method we compute the absolute error between the lifted state in the corresponding Hilbert space obtained by lifting the true trajectory of the Van der Pol oscillator versus the trajectory propagated in time by the corresponding Koopman operator, i.e., KIC and generalized Koopman, respectively.
\begin{figure}[!t]
     \centering
    \subfloat[$5000$ $\col(x,u)$-observables]{%
        \includegraphics[width=0.75\linewidth]{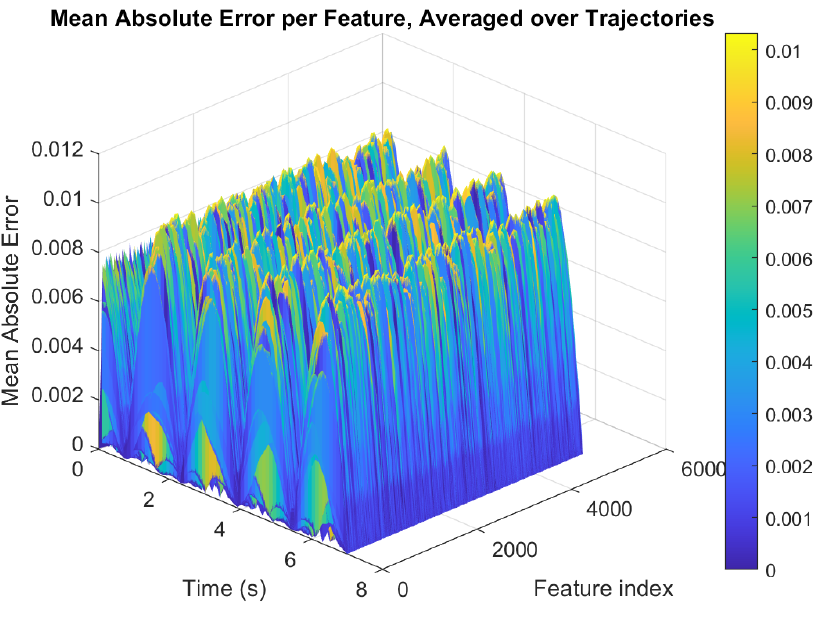}
    }\\
    \subfloat[$10000$ $\col(x,u)$-observables]{%
        \includegraphics[width=0.75\linewidth]{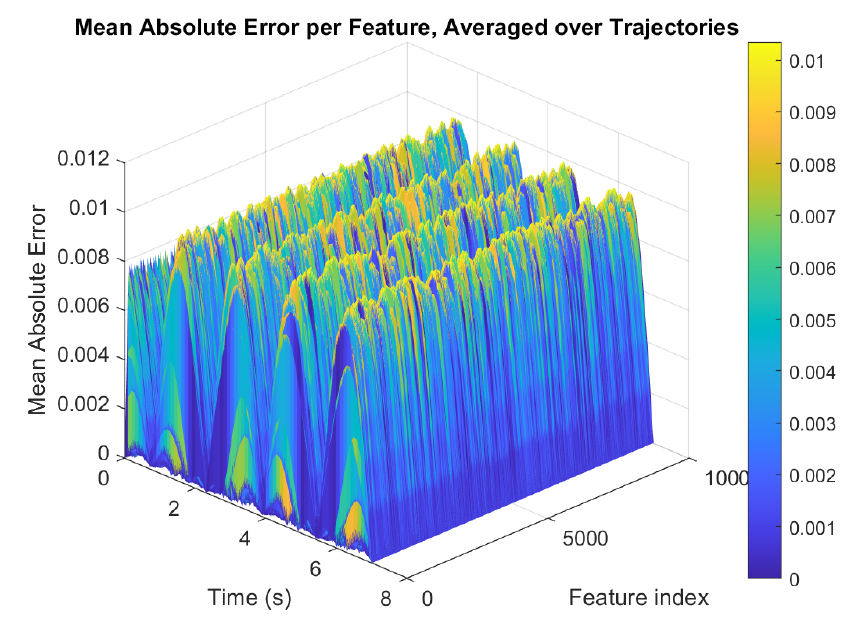}
    }
    \caption{Kernelized KIC method \cite{Brunto_Koopmanu}: Absolute feature error 3D plot averaged over 100 test trajectories.}
    \label{fig:5:1}
    \end{figure}
Figure~\ref{fig:5:1} shows the absolute error plots for the KIC method for $5000$ observables and $10000$ observables, respectively. We observe that the absolute feature error does not improve noticeably despite doubling the dimension of the feature space.
\begin{figure}[!t]
     \centering
    \subfloat[$100$ $x$-observables, $50$ $u$-observables]{%
        \includegraphics[width=0.75\linewidth]{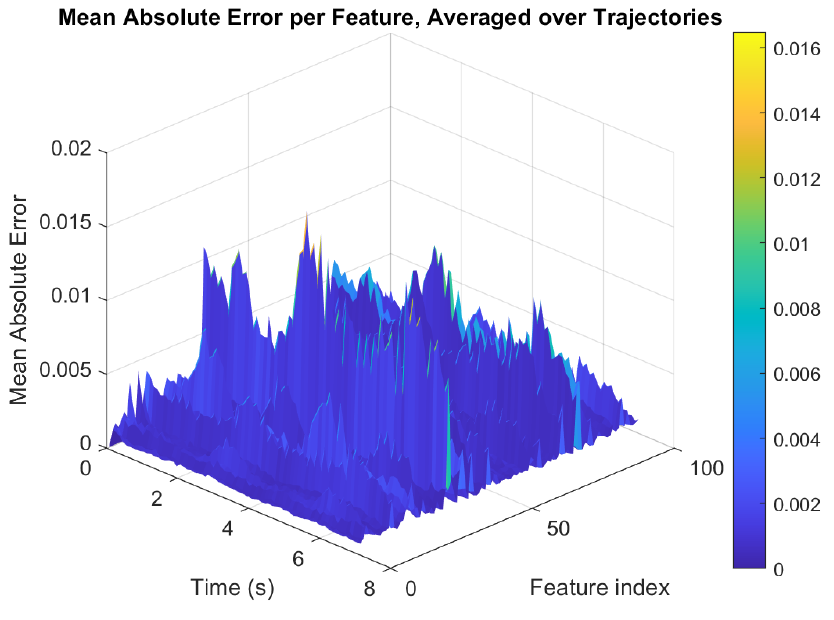}
    }\\
    \subfloat[$200$ $x$-observables, $50$ $u$-observables]{%
        \includegraphics[width=0.75\linewidth]{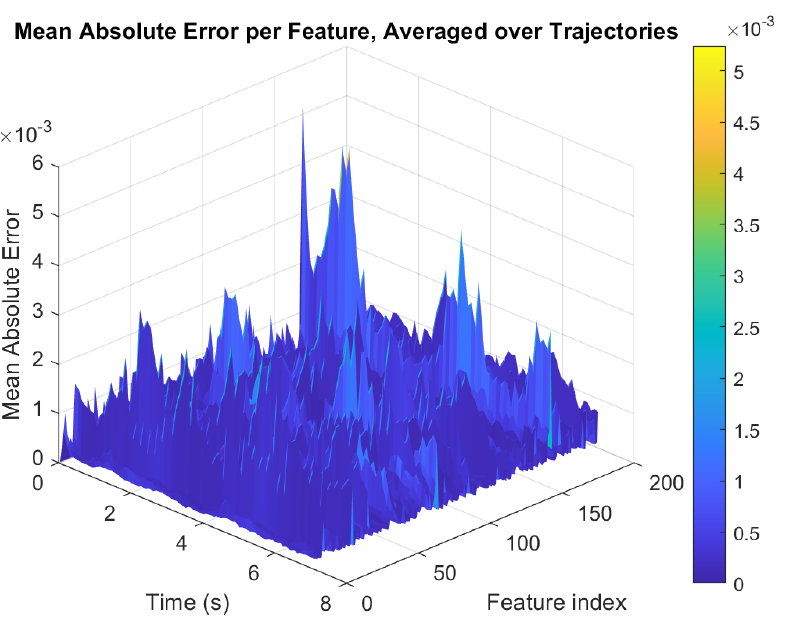}
        }
    \caption{Generalized Koopman: Absolute feature error 3D plot averaged over 100 test trajectories.}
    \label{fig:5:2}
    \end{figure}

Figure~\ref{fig:5:2} shows the absolute error plots for the developed generalized Koopman operator for $100$ and $200$ state observables, respectively, and $50$ input observables in all cases. The dimension of the product Hilbert space is thus $5000$ and $10000$, respectively. As predicted by the theory developed in this paper, we observe that the absolute error consistently improves with the dimension of the feature space, which is also confirmed by the statistical plots shown in Figure~\ref{figsr1}.
\begin{figure}[!t]
\centerline{\includegraphics[width=0.95\columnwidth]{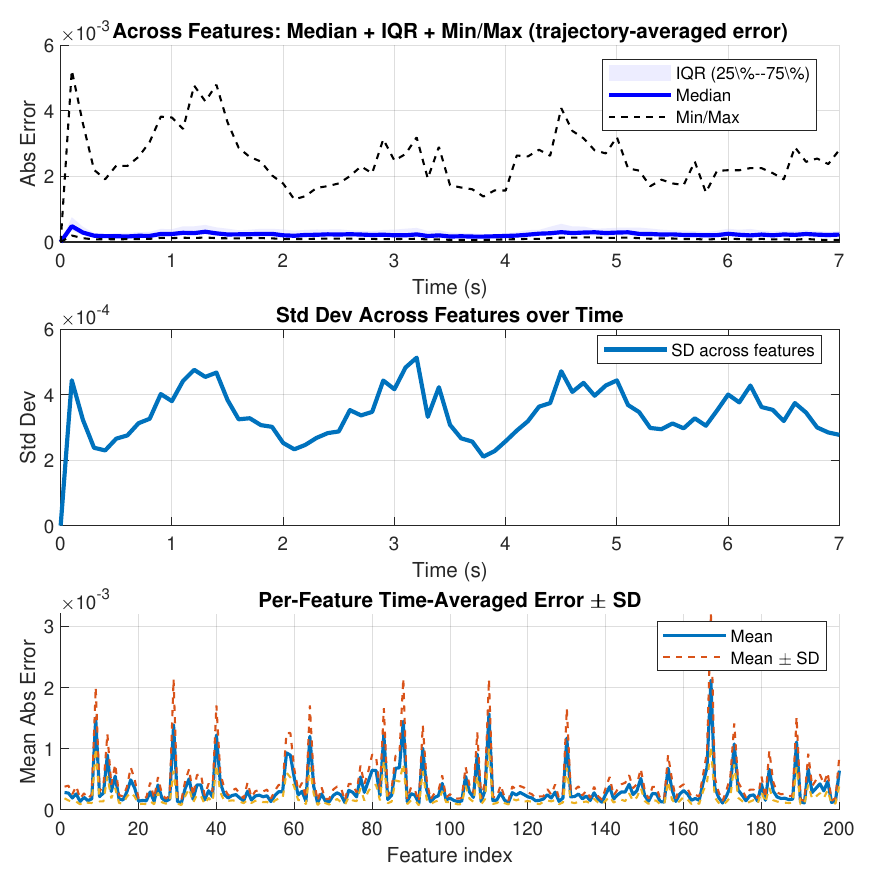}}
\caption{Statistical analysis of the feature error for $200$ $x$-observables, $50$ $u$-observables.}
\label{figsr1}
\end{figure}
\subsection{Soft-robotic manipulator}
As a second illustrative example we consider a non-input-affine nonlinear model motivated by
pneumatically driven soft-robotic manipulators~\cite{yang2024modelbased}. The state vector is $x = [\theta,\,\omega,\,p]^\top$, where $\theta$ is
the bending angle, $\omega=\dot\theta$ is its rate, and $p$ is the internal
chamber pressure. The dynamics are
\begin{align}
  \dot\theta &= \omega, \nonumber\\
  \dot\omega &= -d\,\omega - k\,\theta + \alpha\,p^2, \label{eq:plant}\\
  \dot p      &= a\,\tanh(6\,u)\sqrt{\max(p_s-p,0)}
                 - b\sqrt{\max(p,0)} - c\,p, \nonumber\\
                 y &= \ell\sin\theta,\nonumber
\end{align}
with plant parameters $d=0.8$, $k=2$, $\alpha=0.35$,
$a=3$, $b=1.1$, $c=0.25$ and supply pressure $p_s=5$.

The output of interest is the tip displacement with link length $\ell=1$. Equations are integrated with the forward Euler method at sample time $T_s = 0.05\,\mathrm{s}$. The model in~\eqref{eq:plant} captures three dominant dynamic phenomena of
soft pneumatic actuators: elastic restoring forces
($k\theta$), viscous damping ($d\omega$), pressure-dependent bending torque
($\alpha p^2$), and nonlinear filling/discharge through square-root flow terms.
Its qualitative structure mirrors reduced-order models used in
soft-robot control applications~\cite{yang2024modelbased}.

The state, input and output spaces are defined as
\begin{align*}
  X &= \{x \in \mathbb{R}^3 \mid
    \theta \in [-1.5,\, 1.5],
    \omega \in [-6.0,\, 6.0],
    p \in [0,\, p_s]\},\\
  U &= \{ u \in \mathbb{R} \mid u \in [-1,\, 1] \},\\
  Y &= \{ y \in \mathbb{R} \mid
    y \in [\sin(1.5),\, \sin(1.5)] \approx [-0.997,\, 0.997]\},
\end{align*}
where angles are in radians and $p_s = 5$ is the supply pressure.

We test four Koopman models that share the same $n_z=9$ state observable vector,
defined by the map $\Psi_x: X \to \mathbb{R}^{n_z}$:
\begin{equation}
  z_t := \Psi_x(x_t)
       = \bigl[1,\;\theta,\;\omega,\;p,\;\sin\theta,\;\cos\theta,\;
               p^2,\;\theta\omega,\;\ell\sin\theta\bigr]^\top.
  \label{eq:Psix}
\end{equation}
The dictionary~\eqref{eq:Psix} is constructed to capture the dominant
nonlinear structure of the plant~\eqref{eq:plant}. The output $y = \ell\sin\theta$
is included as $z_9$ so that no reconstruction mapping is required.

Four choices of $\Psi_u: U \to \mathbb{R}^{n_v}$ are compared,
defining $v_t := \Psi_u(u_t)$ as follows: A~--~\emph{Linear}, i.e.,
    the Koopman matrix is identified from the concatenated regressor
    $[z;\, u/\sigma_u]$; B~--~\emph{Bilinear}, i.e.,
    $\Psi_u^{(B)}(u) = [1;\,u]$, yielding an $(n_z{\times}2)$-block product regressor; C~--~\emph{Bilinear and Odd Chebyshev}, i.e.,    $\Psi_u^{(C)}(u) = [1;\,u;\,T_5(u);\,T_7(u);\,T_9(u)]$,   where $T_n$ denotes the $n$-th Chebyshev polynomial of the first kind; D~--~\emph{Bilinear and Tanh bank}, i.e., $\Psi_u^{(D)}(u) = [1;\,u;\,\tanh(4u);\,\tanh(8u)]$, with the input observables explicitly matching the true actuator nonlinearity, but not the actual parameter value.

For models B--D the Koopman matrix $\hat\cK\in\mathbb{R}^{n_z\times n_z n_v}$ is
identified by ridge-regularised least squares on $N_{\mathrm{train}}=20000$
samples:
\begin{equation*}
  K = Z^{+}\Phi^\top\!\left(\Phi\Phi^\top + \lambda I\right)^{-1},
\end{equation*}
where $\Phi_{:,t} = z_t \otimes v_t$ and $Z^{+}_{:,t}=z_{t+1}$,
consistent with Khatri-Rao EDMD \eqref{eq:d:3}. Input features are standardised by their training-set standard deviation
prior to regression to improve numerical conditioning.
Training data are generated by a mixed multi-sine / PRBS excitation
($60\%$ multi-sine, $40\%$ PRBS, 16 sine components) from initial state $x_0 = [0.2; 0; 0.5]$.

The Koopman MPC controllers solve at each sample $t$:
\begin{equation*}
  \min_{U_t} \sum_{i=1}^{N_h}\!
    \bigl[Q_y(z_{t+i,9}-r_{t+i})^2 + R_u u_{t+i-1}^2
    + R_{\Delta u}(\Delta u_{t+i-1})^2\bigr]
\end{equation*}
subject to $U_t\in U^N$, and the Koopman model state $z_{t+i}$ is predicted using the corresponding A--D Koopman model.
The cost is evaluated in the Koopman-lifted space; no reconstruction is needed
because $z_9 = y$. The MPC parameters are $N_h=15$, $Q_y=550$, $R_u=0.05$, $R_{\Delta u}=1$.~For the nonlinear Koopman MPC problems we use \texttt{fmincon} with SQP, warm-started from the previous solution, while for the linear Koopman MPC we use \texttt{quadprog}.

The open-loop validation input is a $125\,\mathrm{s}$, 16-component
multi-sine signal with frequencies uniformly drawn from
$[0.008,\, 0.64]\,\mathrm{Hz}$, random amplitudes and phases,
normalised to $U$, which is independent of the training
excitation. Closed-loop evaluation runs $N_{\mathrm{sim}}{=}900$ steps ($45\,\mathrm{s}$) on a slow piecewise-constant reference profile with smooth tanh-shaped transitions. Both open- and closed-loop simulations start  at zero initial states.

Table~\ref{tab:rmse_ol} reports the free-run RMSE on the output $y$ over the
entire test horizon. The tanh-bank model~D achieves the lowest error, consistent with its
basis matching the true actuator characteristic. The odd Chebyshev model~C is competitive; the bilinear model~B improves over the linear-in-control baseline~A, but leaves residual error.
\begin{table}[h]
\caption{Open-Loop Free-Run RMSE on Output $y$}
\label{tab:rmse_ol}
\centering
\begin{tabular}{lcc}
\toprule
Model & Input observables & RMSE \\
\midrule
A -- linear & $[u]$                             & 0.315789 \\
B -- bilinear          & $[1;\,u]$                         & 0.253991 \\
C -- bil. + odd Chebyshev     & $[1;\,u;\,T_5;\,T_7;\,T_9]$       & 0.229199 \\
D -- bil. + tanh bank         & $[1;\,u;\,\tanh 4u;\,\tanh 8u]$ & 0.219148 \\
\bottomrule
\end{tabular}
\end{table}
\begin{figure}[!t]
\centerline{\includegraphics[width=1\columnwidth]{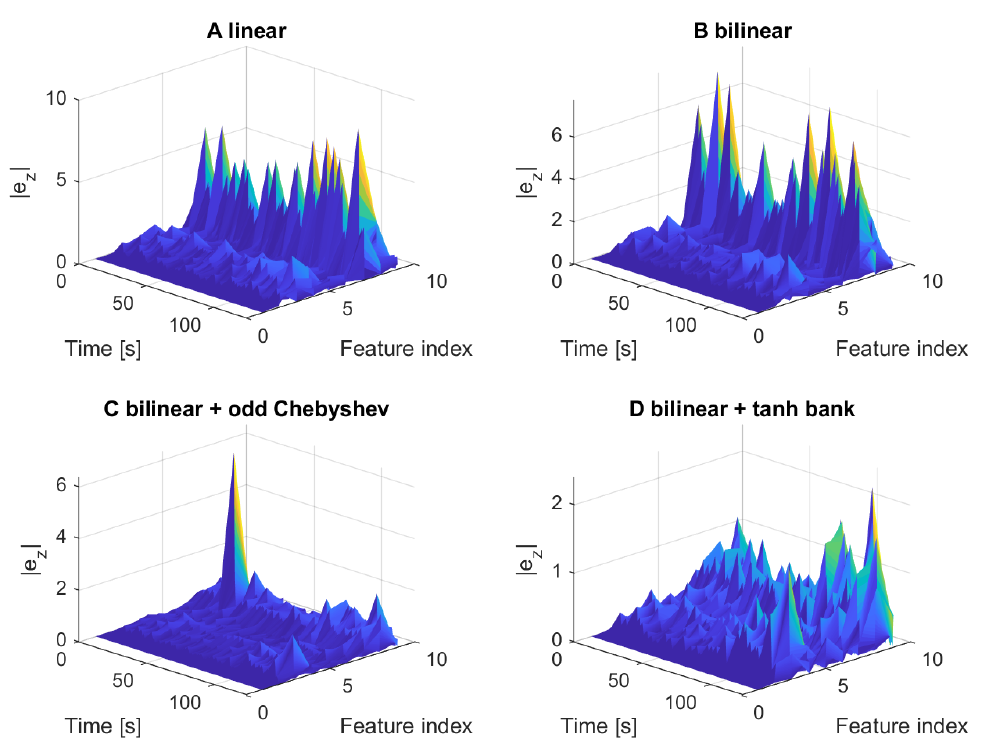}}
\caption{3D absolute error plot across state features for different Koopman embeddings.}
\label{fig:surf_robot}
\end{figure}

Fig.~\ref{fig:surf_robot} shows that the absolute feature error is decreasing with a richer observable dictionary for the input. This shows that for non-input-affine nonlinear systems, going beyond purely bilinear Koopman models can increase accuracy.

Fig.~\ref{fig:closedloop} shows the closed-loop output trajectories for all four Koopman MPC controllers against the reference. Table~\ref{tab:rmse_cl} summarises the closed-loop tracking RMSE, which is consistent with the absolute feature error results.
\begin{table}[h]
\caption{Closed-Loop Tracking RMSE and Median CPU}
\label{tab:rmse_cl}
\centering
\begin{tabular}{lccc}
\toprule
Model & Tracking RMSE & Median CPU [s] \\
\midrule
A -- linear & 0.020783 & 0.000361 \\
B -- bilinear          & 0.016108 & 0.046529 \\
C -- bil. + odd Chebyshev     & 0.011217 & 0.080341 \\
D -- bil. + tanh bank         & 0.007468 & 0.099754 \\
\bottomrule
\end{tabular}
\end{table}
\begin{figure}[t]
  \centering
  \includegraphics[width=0.75\columnwidth]{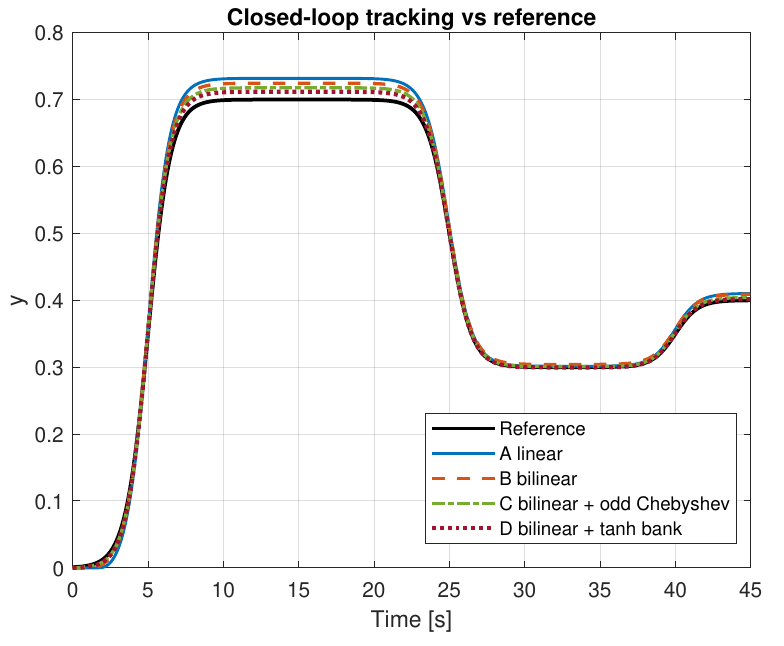}
  \vspace{2pt}
  \includegraphics[width=0.8\columnwidth]{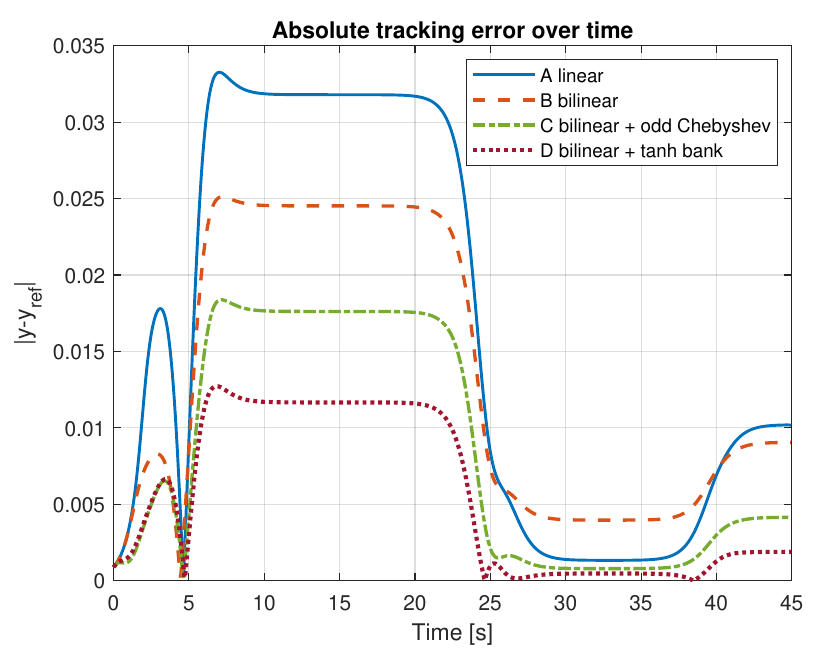}
  \caption{A--D Koopman MPC closed-loop results. Top: reference tracking; Bottom: tracking error.}
  \label{fig:closedloop}
\end{figure}

Note that when the reference comes closer to zero, i.e., $0.3$, a smaller input is required and the input nonlinearity \texttt{tanh} behaves almost linearly. As a result, the linear in $u$ Koopman MPC outperforms the bilinear in $u$ Koopman MPC in this operating range. In contrast, the Koopman MPC based on generalized bilinear models with lifted inputs manages to consistently outperform both the linear and bilinear in $u$ Koopman MPC over the complete tested range of operation.

The per-step CPU times (median averaged over $N_{\mathrm{sim}}{-}1$ calls to \texttt{fmincon/quadprog}) scale roughly with the dimension of the operator. Linear and bilinear in $u$ Koopman MPC implementations remain within real-time requirements of $T_s=0.05\,\mathrm{s}$, while the other Koopman MPC implementations stay close to the required sampling time and could easily reach it via dedicated fast nonlinear MPC solvers (e.g., \texttt{acados}). Efficient numerical implementation is beyond the scope of this paper, however, and it will be considered in future work.

\section{Conclusions}
\label{sec:6}
In this paper we have presented a product Hilbert space approach to generalizing the Koopman operator for nonlinear systems with inputs. We have proven that there exists a well-defined bounded linear operator from the tensor product of two Hilbert spaces spanned by state and input observable functions, respectively, to the state Hilbert space, which yields a generalized, non-measure preserving Koopman operator for systems with inputs. Furthermore, the developed generalized Koopman operator, which leads to an exact infinite-dimensional bilinear model, was utilized to develop a nonlinear fundamental lemma. Methods for data-driven finite-dimensional approximations of the generalized Koopman operator were also indicated along with two illustrative examples that showcased the improved modeling accuracy and control performance of the developed generalized Koopman operator.


\section*{References}
\bibliographystyle{IEEEtran}
\bibliography{Mircea}

\appendices
\section{Proof of Theorem~\ref{thm:relaxed}}
\label{a:0}
Since Assumption~\ref{ass:ia_membership} requires
$\psi_{i,x}\circ F \in \mathcal{H}_x\otimes\mathcal{V}_u
= \mathcal{H}$ for all $i\in\Nset$, the invariance
condition~\eqref{eq:3:8} of Theorem~\ref{thm:1} is
satisfied. Applying Theorem~\ref{thm:1} gives existence
of $\cK = Q_xR^{-1}$ such that
$z_{t+1} = \cK(z_t\otimes v_t)$.
Substituting~\eqref{eq:ia_membership} into the
definition $[Q_x]_{i,(k,l)} :=
\langle\psi_{i,x}\circ F,\,\psi_{k,x}\psi_{l,u}
\rangle_{\mathcal{H}}$ and applying Fubini's theorem yields:
\begin{align}
  [Q_x]_{i,(k,l)}
  &= \int_{X\times U}
     \!\!\psi_{i,x}(F(x,u))
     \psi_{k,x}(x)\,\psi_{l,u}(u)
     \;d(\mu_x\otimes\mu_u) \nonumber\\
  &= \langle a_i,\psi_{k,x}\rangle_{\mathcal{H}_x}
     \underbrace{\int_U\!\psi_{0,u}\,\psi_{l,u}\,d\mu_u}_{[R_u]_{0l}}
     \nonumber\\&\;+\;
     \sum_{j=1}^m
     \langle b_{ij},\psi_{k,x}\rangle_{\mathcal{H}_x}
     \underbrace{\int_U\!\psi_{j,u}\,\psi_{l,u}\,d\mu_u}_{[R_u]_{jl}},
  \label{eq:Qx_entry}
\end{align}
where $[R_u]_{0l},[R_u]_{1l},\ldots,[R_u]_{ml}$
are the entries of the \emph{input Gram matrix}
$R_u \in \mathbb{R}^{(m+1)\times(m+1)}$. Since
$\{\psi_{j,u}\}_{j=0}^m$ are linearly independent
in $L^2(U,\mu_u)$ for $U\subseteq\mathbb{R}^m$ of
positive Lebesgue measure
(Definition~\ref{def:psi_u}), $R_u$ is symmetric
positive definite and hence invertible.
The Gram matrix of the product basis
$\{\psi_{k,x}\psi_{l,u}\}$ is
\begin{equation*}
\begin{split}
  [R]_{(i,l),(k,l')}
  \;&=\; \langle\psi_{i,x}\psi_{l,u},\,\psi_{k,x}\psi_{l',u}
        \rangle_{\mathcal{H}}\\
  \;&=\; \langle\psi_{i,x},\psi_{k,x}\rangle_{\mathcal{H}_x}
        \cdot [R_u]_{ll'}
\end{split}
\end{equation*}
by Fubini's theorem, so $R = R_x\otimes R_u$ and
$R^{-1} = R_x^{-1}\otimes R_u^{-1}$. Both Grams are bounded
since $\|R_x^{-1}\|_{\ell^2\to\ell^2}\leq\alpha^{-1}$
and $R_u^{-1}\in\mathbb{R}^{(m+1)\times(m+1)}$ is finite.
From~\eqref{eq:Qx_entry}, it holds that
\begin{equation*}
  [Q_x]_{i,(k,l)}
  \;=\; \langle a_i,\psi_{k,x}\rangle_{\mathcal{H}_x}
        [R_u]_{0l}
        \;+\; \sum_{j=1}^m
              \langle b_{ij},\psi_{k,x}\rangle_{\mathcal{H}_x}
              [R_u]_{jl},
\end{equation*}
or in matrix form:
\begin{equation*}
  Q_x \;=\; A_x\otimes(e_0^\top R_u)
       \;+\; \sum_{j=1}^m
             N_{j,x}\otimes(e_j^\top R_u),
\end{equation*}
where $[A_x]_{ik} :=
\langle a_i,\psi_{k,x}\rangle_{\cH_x}$,
$[N_{j,x}]_{ik} :=
\langle b_{ij},\psi_{k,x}\rangle_{\cH_x}$.
Right-multiplying by $R^{-1} =
R_x^{-1}\otimes R_u^{-1}$ and applying the
mixed product property $(A\otimes B)(C\otimes D) = AC\otimes BD$ yields
\begin{equation*}
\begin{split}
  \cK &= Q_x(R_x^{-1}\otimes R_u^{-1})\\
    &= A_xR_x^{-1}\otimes(e_0^\top R_u R_u^{-1})
    \;+\; \sum_{j=1}^m
          N_{j,x}R_x^{-1}\otimes(e_j^\top R_u R_u^{-1}).
\end{split}
\end{equation*}
Since $R_uR_u^{-1} = I_{m+1}$, we have
$e_j^\top R_uR_u^{-1} = e_j^\top$, so the
$j=0$ block gives $\cA: = A_xR_x^{-1}$ and
each $j\geq 1$ block gives $\cN_j := N_{j,x}R_x^{-1}$,
yielding the factorization $\cK = \cA\otimes e_0^\top
+ \sum_{j=1}^m\cN_j\otimes e_j^\top$.
Since $v_t = \operatorname{col}(\psi_{0,u}(u_t),\ldots,\psi_{m,u}(u_t))
= \operatorname{col}(1,u_{1,t},\ldots,u_{m,t})$
(Definition~\ref{def:psi_u}), we have
$z_t\otimes v_t =
\operatorname{col}(z_t,\,u_{1,t}z_t,\,\ldots,\,u_{m,t}z_t)$,
and hence
\begin{equation*}
  z_{t+1} \;=\; \cK(z_t\otimes v_t)
  \;=\; \cA z_t + \sum_{j=1}^m \cN_jz_t\,u_{j,t}.
\end{equation*}
Boundedness of $\cA = A_xR_x^{-1}$ and
$\cN_j = N_{j,x}R_x^{-1}$ follows from
$\|R_x^{-1}\|_{\ell^2\to\ell^2}\leq\alpha^{-1}$
and boundedness of $A_x$, $N_{j,x}$, where each row
$(\langle a_i,\psi_{k,x}\rangle_{\cH_x})_{k=1}^\infty
\in\ell^2$ and
$(\langle b_{ij},\psi_{k,x}\rangle_{\cH_x})_{k=1}^\infty
\in\ell^2$ by Bessel's inequality
(Remark~\ref{rem:K_bounded} with $\cH_u = \cV_u$).
Boundedness of $\cK$ then follows from the factorization
$\cK = \cA\otimes e_0^\top +
\sum_{j=1}^m\cN_j\otimes e_j^\top$.
\section{Proof of Proposition~\ref{prop:1y}}
\label{a:1}
Since each $\psi_{i,y}\in\cH_y$, $i\in\Nset$, we can expand each $\psi_{i,y}$ as $\psi_{i,y}=\sum_{j=1}^\infty\langle\psi_{i,y},\phi_j\rangle_{\cH_y}\phi_j$. Then we have that
\begin{equation}
\label{eq:4:5}
\begin{split}
\Psi_y\circ h=\begin{bmatrix}\psi_{1,y}\circ h\\\psi_{2,y}\circ h\\\vdots\end{bmatrix}&=\begin{bmatrix}\sum_{j=1}^\infty\langle\psi_{1,y},\phi_j\rangle_{\cH_y}\phi_j\circ h\\\sum_{j=1}^\infty\langle\psi_{2,y},\phi_j\rangle_{\cH_y}\phi_j\circ h\\\vdots\end{bmatrix}\\
&=\sum_{j=1}^\infty\begin{bmatrix}\langle\psi_{1,y},\phi_j\rangle_{\cH_y}\\\langle\psi_{2,y},\phi_j\rangle_{\cH_y}\\\vdots\end{bmatrix}\phi_j\circ h.
\end{split}
\end{equation}
Since $\phi_j\circ h\in\cH_x$ for $j\in\Nset$ and $\{\psi_{i,x}\}_{i=1}^\infty$ is an orthonormal basis in $\cH_x$, we can expand $\phi_j\circ h$ as
\begin{equation}
\label{eq:4:6}
\phi_j\circ h = \sum_{i=1}^\infty\langle\phi_j\circ h,\psi_{i,x}\rangle_{\cH_x}\psi_{i,x}.
\end{equation}
Substituting \eqref{eq:4:6} in \eqref{eq:4:5} yields
\[
\begin{split}
&\Psi_y[h(x)]=\sum_{j=1}^\infty\begin{bmatrix}\langle\psi_{1,y},\phi_j\rangle_{\cH_y}\\\langle\psi_{2,y},\phi_j\rangle_{\cH_y}\\\vdots\end{bmatrix}\phi_j\circ h(x)\\
&=\sum_{j=1}^\infty\begin{bmatrix}\langle\psi_{1,y},\phi_j\rangle_{\cH_y}\\\langle\psi_{2,y},\phi_j\rangle_{\cH_y}\\\vdots\end{bmatrix}\sum_{i=1}^\infty\langle\phi_j\circ h,\psi_{i,x}\rangle_{\cH_x}\psi_{i,x}(x)=\\
&\sum_{j=1}^\infty\begin{bmatrix}\langle\psi_{1,y},\phi_j\rangle_{\cH_y}\\\langle\psi_{2,y},\phi_j\rangle_{\cH_y}\\\vdots\end{bmatrix}[\langle\phi_j\circ h, \psi_{1,x}\rangle_{\cH_x}, \langle\phi_j\circ h, \psi_{2,x}\rangle_{\cH_x}, \ldots]\cdot\\&\cdot\begin{bmatrix}\psi_{1,x}(x)\\\psi_{2,x}(x)\\\vdots\end{bmatrix}=\cC\Psi_x(x).
\end{split}
\]
\section{Proof of Theorem~\ref{thm:1y}}
\label{a:2}
From \eqref{eq:3:cy} we obtain
\[\begin{bmatrix}\phi_1\circ h\\\phi_2\circ h\\\vdots\end{bmatrix}=T_y^{-1}\begin{bmatrix}\psi_{1,y}\circ h\\\psi_{2,y}\circ h\\\vdots\end{bmatrix},\]
which, together with \eqref{eq:3:8y} implies $\phi_i\circ h\in\cH_x$ for all $i\in\Nset$. Then, from Proposition~\ref{prop:1y} we have that there exists a matrix $\cC$ such that \eqref{eq:3:9y} holds. Multiplying both sides in \eqref{eq:3:9y} with a row vector function $[\psi_{1,x},\psi_{2,x},\ldots]$ and integrating each element of the resulting matrices of functions over $X$, via a similar reasoning as in the proof of Theorem~\ref{thm:1}, we obtain
\[\begin{split}
&\begin{bmatrix}\langle\psi_{1,y}\circ h,\psi_{1,x}\rangle_{\cH_x} &\langle\psi_{1,y}\circ h,\psi_{2,x}\rangle_{\cH_x}&\ldots\\\langle\psi_{2,y}\circ h,\psi_{1,x}\rangle_{\cH_x}&\langle\psi_{2,y}\circ h,\psi_{2,x}\rangle_{\cH_x}&\ldots\\\vdots&\vdots&\ddots\end{bmatrix}\\&\quad =\cC\begin{bmatrix}\langle\psi_{1,x},\psi_{1,x}\rangle_{\cH_x} &\langle\psi_{1,x},\psi_{2,x}\rangle_{\cH_x}&\ldots\\\langle\psi_{2,x},\psi_{1,x}\rangle_{\cH_x}&\langle\psi_{2,x},\psi_{2,x}\rangle_{\cH_x}&\ldots\\\vdots&\vdots&\ddots\end{bmatrix},
\end{split}
\]
and thus, $Q_y=\cC R_x$. Since $\{\psi_{i,x}\}_{i=1}^\infty$ is a Riesz basis in $\cH_x$ and thus, its Gram matrix $R_x$ is invertible, yields $\cC=Q_yR_x^{-1}$.

\begin{IEEEbiography}[{\includegraphics[width=1in,height=1.25in,clip,keepaspectratio]{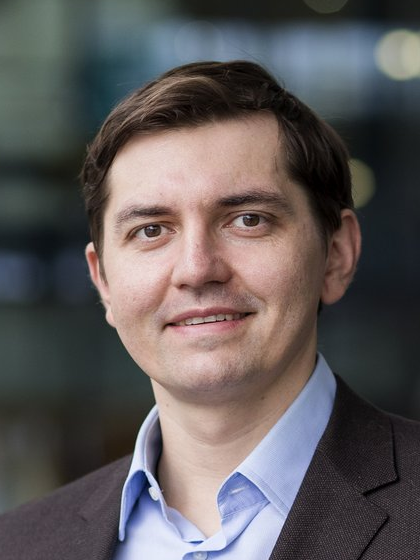}}]{Mircea Lazar}
	was born on March $\text{4}^{\text{th}}$, 1978 in Iași, Romania and is an Associate Professor in Constrained Control of Complex Systems at the Department of Electrical Engineering, Eindhoven University of Technology, Eindhoven, the Netherlands. His research interests cover stability theory, Lyapunov functions, artificial intelligence for prediction and control, and constrained control of non-linear and hybrid systems, including model predictive control. His research is driven by control problems in power systems, power electronics, high-precision mechatronics, automotive and biological systems. Dr. Lazar received the European Embedded Control Institute PhD Award in 2007 for his PhD dissertation. He supervised 15 PhD researchers (2 received the cum laude distinction). Dr. Lazar is currently appointed as Chair of the IFAC Technical Committee (TC) 2.3 Non-Linear Control Systems and as Associate Editor of IEEE Transactions on Automatic Control.
\end{IEEEbiography}

\end{document}